\documentclass[12pt,amstex,amssymb]{amsart}
\usepackage{amsmath,amsthm,amsfonts,amssymb,amscd}
\usepackage[latin1]{inputenc}
\usepackage[all]{xy}
\usepackage{graphicx}
\usepackage{amsmath}
\usepackage{amsthm}
\usepackage{amsfonts}
\usepackage{amssymb}%

\numberwithin{equation}{section}
\graphicspath{ {images/} }

\usepackage{amsmath}
\usepackage{amssymb}
\usepackage{amsthm}
\def\fa{{\mathcal{F}}}
\def\O{{\mathcal{O}}}

\def\M{{\mathcal{M}}}
\def\D{{\mathcal{D}}}
\def\U{{\mathcal{U}}}
\def\B{{\mathcal{B}}}

\def\la{{\lambda}}
\def\bn{{\mathbb{N}}}
\def\bz{{\mathbb{Z}}}
\def\ve{{\varepsilon}}
\def\bq{{\mathbb{Q}}}

\def\rg{{\rangle}}

\title[On first integrals admitting asymptotic expansion]{On first integrals admitting asymptotic expansion for holomorphic foliations in dimension two}

\author{F. Reis}

\address{F. Reis: Instituto de Matemática - Universidade Federal do Rio de Janeiro, CP. 68530-Rio
de Janeiro-RJ, 21945-970 - Universidade Federal do Espírito Santo - CEP 29932-540 - Brazil
}
\email{fernando.reis@ufes.br}

\date{}

\begin{document}

\begin{abstract}
In this paper we study the dynamics of a holomorphic vector field near a singular point in dimension two using asymptotic expansion techniques. We consider a holomorphic vector field admitting first integrals in small sectors with nonzero asymptotic expansion. Under some natural hypotheses we prove the existence of a holomorphic first integral.
\end{abstract}

\maketitle

\newtheorem{Theorem}{Theorem}[section]
\newtheorem{Corollary}{Corollary}[section]
\newtheorem{Proposition}{Proposition}[section]
\newtheorem{Lemma}{Lemma}[section]
\newtheorem{Claim}{Claim}[section]
\newtheorem{Definition}{Definition}[section]
\newtheorem{Example}{Example}[section]
\newtheorem{Remark}{Remark}[section]

%%%%%%%%%%%%%%%%%%%%%% Paper-specific definitions %%%%%%%%%%%%%%%%%
\newcommand\virt{\rm{virt}}
\newcommand\SO{\rm{SO}}
\newcommand\G{\varGamma}
\newcommand\Om{\Omega}
\newcommand\Kbar{{K\kern-1.7ex\raise1.15ex\hbox to 1.4ex{\hrulefill}}}
\newcommand\codim{\rm{codim}}
\renewcommand\:{\colon}
\newcommand\s{\sigma}
\def\vol#1{{|{\bfS}^{#1}|}}

\def\fa{{\mathcal F}}
\def\H{{\mathcal H}}
\def\O{{\mathcal O}}
\def\P{{\mathcal P}}
\def\L{{\mathcal L}}
\def\C{{\mathcal C}}
\def\Z{{\mathcal Z}}

\def\M{{\mathcal M}}
\def\N{{\mathcal N}}
\def\R{{\mathcal R}}
\def\ea{{\mathcal e}}
\def\Oa{{\mathcal O}}
\def\ee{{\bfE}}

\def\A{{\mathcal A}}
\def\B{{\mathcal B}}
\def\H{{\mathcal H}}
\def\V{{\mathcal V}}
\def\U{{\mathcal U}}
\def\al{{\alpha}}
\def\be{{\beta}}
\def\ga{{\gamma}}
\def\Ga{{\Gamma}}
\def\om{{\omega}}
\def\Om{{\Omega}}
\def\La{{\Lambda}}
\def\ov{\overline}
\def\dd{{\bfD}}
\def\pp{{\bfP}}

\def\nn{{\mathbb N}}
\def\zz{{\mathbb Z}}
\def\bq{{\mathbb Q}}
\def\bp{{\mathbb P}}
\def\bd{{\mathbb D}}
\def\bh{{\mathbb H}}
\def\te{{\theta}}
\def\rr{{\mathbb R}}
\def\bb{{\mathbb B}}

\def\pp{{\mathbb P}}

\def\dd{{\mathbb D}}
\def\zz{{\mathbb Z}}
\def\qq{{\mathbb Q}}

\def\hh{{\mathbb H}}
\def\nn{{\mathbb N}}

\def\LL{{\mathbb L}}

\def\co{{\mathbb C}}
\def\qq{{\mathbb Q}}
\def\na{{\mathbb N}}
\def\esima{${}^{\text{\b a}}$}
\def\esimo{${}^{\text{\b o}}$}
\def\rg{\rangle}
\def\ro{{\rho}}
\def\lV{\left\Vert}
\def\rV{\right\Vert }
\def\lv{\left|}
\def\rv{\right| }
\def\Sa{{\mathcal S}}
\def\D{{\mathcal D  }}

\def\si{{\bf S}}
\def\ve{\varepsilon}
\def\vr{\varphi}
\def\lV{\left\Vert }
\def\rV{\right\Vert}
\def\lv{\left| }
\def\rv{\right|}
\def\Range{\rm{{R}}}
\def\vol{\rm{{Vol}}}
\def\ind{\rm{{i}}}

\def\Int{\rm{{Int}}}
\def\Dom{\rm{{Dom}}}
\def\supp{\rm{{supp}}}
\def\Aff{\rm{{Aff}}}
\def\Exp{\rm{{Exp}}}
\def\Hom{\rm{{Hom}}}
\def\codim{\rm{{codim}}}
\def\cotg{\rm{{cotg}}}
\def\dom{\rm{{dom}}}
\def\Sa{\mathcal{{S}}}

\def\VIP{\rm{{VIP}}}
\def\argmin{\rm{{argmin}}}
\def\Sol{\rm{{Sol}}}
\def\Ker{\rm{{Ker}}}
\def\Sat{\rm{{Sat}}}
\def\diag{\rm{{diag}}}
\def\rank{\rm{{rank}}}
\def\Sing{\rm{{Sing}}}
\def\sing{\rm{{sing}}}
\def\hot{\rm{{h.o.t.}}}

\def\Fol{\rm{{Fol}}}
\def\grad{\rm{{grad}}}
\def\id{\rm{{id}}}
\def\Id{\rm{{Id}}}
\def\sep{\rm{{Sep}}}
\def\Aut{\rm{{Aut}}}
\def\Sep{\rm{{Sep}}}
\def\Res{\rm{{Res}}}
\def\ord{\rm{{ord}}}
\def\h.o.t.{\rm{{h.o.t.}}}
\def\Hol{\rm{{Hol}}}
\def\Diff{\rm{{Diff}}}
\def\SL{\rm{{SL}}}

\tableofcontents

\section{Introduction and main results}
A holomorphic 1-form $\omega(x,y) = A(x,y)dx + B(x,y)dy$ defined in a neighborhood $U$ of $0 \in \mathbb{C}^2$ defines a holomorphic foliation $\fa_{\omega}$ in $U$ with singular set $sing(\fa_{\omega}) = sing(\omega)$. For our purposes we may assume that $sing(\omega) = \{0\}$. Indeed we are concerned with the corresponding \textit{germ} of foliation at the origin. Given such a germ $\fa$ a \textit{separatrix} of $\fa$ is the germ of an irreducible analytic curve $\Gamma$ with $0 \in \Gamma$ and which is invariant by $\fa$. In this case $\Gamma \setminus \{0\}$ is a leaf of $\fa$. The existence of separatrix is proved in \cite{Ca-Sad-LN generalizedCurve}. Foliations with a finite number of separatrices are called \textit{non-dicritical}.

Motivated by the case of real planar vector fields we want to investigate the connections between the dynamics of $\fa$ near its separatrices and the global behavior of $\fa$. Our approach is based on the notion of asymptotic expansion of one-variable holomorphic functions. More precisely we investigate minimal conditions under which the existence of suitable local first integrals in suitable open sets near the separatrices assure the existence of a holomorphic first integral in a neighborhood of the origin. For this sake we shall work with the notion asymptotic expansion of a holomorphic function. 

The classical theory of asymptotic expansion is a well developed subject. There are many classic books on the subject (for instance \cite{Balser1, Balser2, Costin, Ramis3, Wasow}). The aim of this paper is to study asymptotic expansion associated to holomorphic dynamical systems. 

Given a germ $\mathcal{F}$ of foliation at $0 \in \mathbb{C}^2$ and a separatrix $\Gamma$ of $\mathcal{F}$ we shall say that a pair $(\mathcal{F}, \Gamma)$ admits a \textit{sectorial first integral} if given any representative $\mathcal{F}_{U}$ of $\mathcal{F}$  in a neighborhood $U$ of $0 \in \mathbb{C}^2$, for any transverse section $\Sigma$ with $p = \Sigma \cap \Gamma \neq \{0\}$, there exist a sector $S \subset \Sigma$ with vertex at $p$ which is invariant by $\mathcal{F}_{U}$ and a holomorphic function $\varphi: S \rightarrow \mathbb{C}$ such that $\varphi $ is constant in the trace of each leaf of $\mathcal{F}$ in $U$. The sectorial first integral is \textit{moderate} if $\varphi $ admits a nonzero asymptotic expansion in $S$. 
If  $\mathcal{F}$ is a germ of foliation, admitting a holomorphic first integral $f:U \rightarrow \mathbb{C}$ given by the restriction $\varphi = f|_{\Sigma}$, then $\fa$ is non-dicritical and for each separatrix $\Gamma$ of $\mathcal{F}$ the pair $(\mathcal{F},\Gamma)$ admits a moderate sectorial first integral. The converse of this fact is stated below:

\begin{Theorem} \label{Main}
Let $\mathcal{F}$ be a non-dicritical germ of holomophic foliation at $0 \in \mathbb{C}^2$. Assume that $\mathcal{F}$ is a generalized curve and some pair $(\mathcal{F}, \Gamma)$ admits a moderate sectorial first integral. Then $\mathcal{F}$ admits a holomorphic first integral in a neighborhood of $0 \in \mathbb{C}^2$.
\end{Theorem} 

We refer to \cite{Ca-Sad-LN generalizedCurve} for the notion of generalized curve. As for the moment we say that the desingularization of $\fa$ by quadratic blow-ups produces no saddle-node singularities, i.e, only singularities with non-singular linear part. Let us introduce another important notion we use. A separatrix $\Gamma$ of $\mathcal{F}$ is called \textit{non-weak} if after the reduction of singularities of $\mathcal{F}$, every point of the exceptional divisor can be connected to $\Gamma$ by a sequence of projective lines starting at $\Gamma$ and such that every time we reach a corner, we arrive either on a separatrix of a non-degenerated singularity or on a strong (non-weak) manifold of a saddle-node. We refer to section §6 for further details on the resolution of singularities process.

With this notion we can extend Theorem \ref{Main} as follow:

\begin{Theorem} \label{Main2}
Let $\mathcal{F}$ be a non-dicritical germ of holomophic foliation at $0 \in \mathbb{C}^2$. Assume that for some non-weak separatrix $\Gamma$ of $\mathcal{F}$ the pair $(\mathcal{F}, \Gamma)$ admits a moderate sectorial first integral. Then $\mathcal{F}$ admits a holomorphic first integral in a neighborhood of $0 \in \mathbb{C}^2$.
\end{Theorem} 
The hypothesis of generalized curve can be removed if we consider all separatrices of $\mathcal{F}$. Indeed,
\begin{Theorem} \label{Main3}
Let $\mathcal{F}$ be a non-dicritical germ of holomophic foliation at $0 \in \mathbb{C}^2$. Assume that for every separatrix $\Gamma$ of $\fa$ the pair $(\fa,\Gamma)$ admits a moderate sectorial first integral then $\fa$ admits a holomorphic first integral.
\end{Theorem}
As an application of our results we will give a proof of the following theorem of Mattei-Moussu:
\begin{Theorem} [Mattei-Moussu \cite{mattei-moussu}, pg. 472, Theorem A]\label{Cor3} 
Let $\mathcal{F}$ be a germ of holomophic foliation at $0 \in \mathbb{C}^2$. Suppose that $\fa$ admits a nonzero formal first integral. Then $\fa$ admits a holomorphic first integral.
\end{Theorem}

In section §\ref{ExampleSection} we discuss the above hypotheses in some examples. The paper is organized as follows: In section §2 we give definitions and some basic results in asymptotic theory in one dimensions. In section §3 we study germs of diffeomorphisms admitting an invariant sector. In section §4 we prove the existence of a formal normal form for formal power series. Section §\ref{FinitSubgroup} is devoted to the study of the subgroups of diffeomorphisms that are invariant by functions with asymptotic expansion. Section §\ref{ReviewFoliations}, we give some basics notions of holomorphic foliations and resolution of singularities. In the next section, we turn our attention to some importants examples that justify the hypothesis of main results. In section §\ref{DulacSection} we explore an important tool called Dulac correspondence. As an aplication of this tool we can transpont the moderate sectorial first integral. This fact is fundamental to prove Theorem \ref{Main} and Theorem \ref{Main2}. In section §\ref{SectionMain} we present the proof Theorem \ref{Main}. In section §\ref{SaddlenodeSection} we study the saddle-node case. We give the notion of asymptotic expansion in dimension two. Furthermore, we use the Hukuara-Kimura-Matuda theorem \cite{hukuara-kimura-matuda} in order to study the strong separatrix and the presence of the moderate sectorial first integral. In section §\ref{Proof of theorem2} we prove a fundamental proposition to demonstrate Theorem \ref{Main2}. In section §\ref{Section:GenCurve} we prove Theorem \ref{Main3}. In the last section, we give an application of the techniques developed in the previous sections. We present a proof for Theorem \ref{Cor3}.

\section{A review of asymptotic expansion in dimension one} \label{ReviewSection}
Let us recall the definition of asymptotic expansion in dimension one.

\begin{Definition}{\rm 
Given $R \in \mathbb{R}_{+}$ and $\alpha_{1},\alpha_{2} \in \mathbb{R}$, a \textit{sector $S$ with vertex at the origin}  in $\mathbb{C}$ is a subset defined by
\[
S = \{z \in \mathbb{C} : 0 < \mid z \mid < R, \alpha_{1} < arg(z) < \alpha_{2}\}.
\]
}
\end{Definition}

\begin{Definition}[\cite{Wasow} pg. 32]{\rm 
A holomorphic function $\varphi $ is said to \textit{admits a formal power series $\hat{\varphi }(z) = \sum^{\infty}_{n=0}a_{j}z^{j}$ as asymptotic expansion} on a sector $S\subset \mathbb{C}$ if for all proper sub-sector $S'\subset S$ and all $k \in \mathbb{N}$, there exist a constant $A_{k} > 0$ such that
\[
\mid\mid \varphi(z) - \sum^{k-1}_{n=0}a_{j}z^{j}\mid\mid \leq A_{k} \mid \mid z \mid \mid^{k}
\]
for all $z \in S'$.
}
\end{Definition}
The existence of an asymptotic expansion is guaranteed by the following result:

\begin{Theorem} [cite{Wasow} pg. 40]\label{coefasymp2}
Let $\varphi $ be holomorphic in a sector $S$. If all the limits 
\[
\varphi _{r} = \lim_{z \rightarrow 0} \varphi^{(r)}(z), 
\]
exist for every $r = 0,1,... $ then $\varphi $ admits $\sum_{r=0}^{\infty}\frac{\varphi_{r}}{r!} z^{r}$ as asymptotic expansion in $S$.
\end{Theorem}
We will summarize the main theorems on asymptotic expansion in a single result as follows. 
\begin{Theorem} [\cite{Wasow} pg. 33]\label{BusAsymp}
Let $\varphi , \psi$ be holomorphic functions that admit, respectively, $\sum_{r=0}^{\infty}a_{r}z^{r}$ in a sector $S$ and $\sum_{r=0}^{\infty}b_{r}z^{r}$ in a sector $T$, as asymptotic expansion. Then 
\\
\\
\emph{(I)} $a_{r} = \frac{1}{r!} \lim_{z \rightarrow 0} \varphi ^{(r)}(z)$, with $z$ belongs to proper subsectors.
\\
\\
\emph{(II)} If $S = D(0,r) - \{0\}$ with radius $r>0$ and vertex at $0 \in \mathbb{C}$ then $\varphi$ is holomorphic in $D(0,r)$ and $a_r = \frac{\varphi^{(r)}(0)}{r!}$.
\\
\\
\emph{(III)} $\varphi$ admits at most one asymptotic series representation in $S$.
\\
\\
\emph{(IV)} $\alpha \varphi + \beta \psi$ admits $\sum_{r=0}^{\infty}(\alpha a_{r} + \beta b_{r}) z^{r}$ as asymptotic expansion. \\
\\
\emph{(V)} The product $\varphi \psi$ admits $\left(\sum_{r=0}^{\infty}a_{r}z^{r}\right) \left(  \sum_{r=0}^{\infty}b_{r}z^{r}\right) = \sum_{r=0}^{\infty}c_{r}z^{r} $ as asymptotic expansion, with \[
c_{r} =\sum_{j=0}^{r} a_{j}b_{r-j}.
\]
\\
\\
\emph{(VI)} if $a_{0} \neq 0$, then 
$
\frac{1}{\varphi(z)}
$ admits $\sum_{0}^{\infty}c_{t}z^{t}$ where,
\[
\left(\sum_{r=0}^{\infty}a_{r}z^{r}\right)\left(\sum_{r=0}^{\infty}c_{t}z^{t}\right) = 1.
\]
\\
\\
\emph{(VII)} The composition $\psi \circ \varphi(z)$ admits $\sum_{0}^{\infty}d_{t}z^{t}$ as asymptotic expansion, where the coefficients $d_t$ are obtained by formal insertion of the series for $\varphi(z)$ into the series for $\psi(u)$ followed by collection of like powers of $z$. 
\\
\\
\emph{(VIII)} If $S = \{ 0 < \mid z \mid < z_0, \: \theta_1 \leq arg(z) \leq \theta_{2}\}$ with $\theta_{2} > \theta_1$, then the derivative $\varphi'(z)$ admits $\sum_{r=0}^{\infty}ra_{r}z^{r-1}$ as asymptotic expansion in $S^{*} = \{ 0 < \mid z \mid < z_0, \: \theta_1 < \theta_{1}^{*} \leq arg(z) \leq \theta_{2}^{*} < \theta_{2} \}$.
\end{Theorem}
Finally, the next result will be useful later in this paper. 
\begin{Theorem} [\textbf{Borel-Ritt}, \cite{Wasow} pg. 40] \label{Borel-Ritt} 
Let $\hat{\varphi} = \sum a_r z^r \in \mathbb{C}[[z]]$ be a formal power series and $S$ a sector in $\mathbb{C}$. Then there exists a function $\varphi$ holomorphic in $S$ such that $\varphi$ admits $\hat{\varphi}$ as asymptotic expansion in $S$.
\end{Theorem}

\begin{Example}[\cite{Loday}, pg. 8] \label{Exe3}\rm{  
An interesting example of function with asymptotic expansion is given by a classical differential equation
\begin{eqnarray}
x^2\frac{dy}{dx} + y = x. \label{EulerEq}
\end{eqnarray}
The \ref{EulerEq} equation is called \textit{Euler equation}. The divergent series $\widetilde{E}(x) = \sum_{n \leq 0}(-1)^n n! x^{n+1}$ is the unique power series solution of the Euler equation. The series $\widetilde{E}$ is called \textit{Euler series} and it is divergent for all $x \neq 0$. The function
\[
E(x) = \int_{0}^{x} \exp\left(-\frac{1}{t} + \frac{1}{x}\right)\frac{dt}{t} = \int_{0}^{+\infty} \frac{e^{-\frac{\\xi}{x}}}{1 + \xi}d\xi.
\]
admits $\widetilde{E}(x)$ as asymptotic expansion on the sector $S = \{x \in \mathbb{C}; Re(x)>0 \}.$
}
\end{Example}

\section{Germs of diffeomorphisms and invariant sectors}
We denote by $Diff(\mathbb{C},0)$ the group of germs of diffeomorphism fixing the origin. Consider $f \in Diff(\mathbb{C},0)$ given by $f(z) = \lambda z +a_{k+1} z^{k+1} + \dots$, $k \geq 1,
\, \lambda \in \mathbb{C}^*$. We have the following classification:
\\
\\
\noindent (i) \textit{Hyperbolic Case:}\quad $|\la| \ne 1$; \\
\\
\noindent (ii) \textit{Elliptic Case:}\quad $|\la|=1$, \,\, $\lambda ^{k} \neq 1$, for all $k \in \mathbb{N} \setminus \{0\}$;
\\
\\
\noindent (iii)\, \textit{Parabolic Case:}\quad $\la^k=1$
for some $k \in \bn \setminus \{0\}$.
\\

\begin{Definition} {\rm
Let $f$ be a germ of holomorphic diffeomorphism fixing $0$. We say that $f$ \textit{admits an proper invariant sector} $S \subsetneqq \mathbb{C}^{*}$, if for a representative $f_{U} :U \rightarrow f(U)$ of $f$ there is a sector $S \subset U$ with vertex at $0$ such that $f(S) \subset S$. In this case we also say that $S$ is \textit{$f$-invariant}. 
}
\end{Definition}
Invariant sectors always exist in the case of germs tangent to the identity.

\begin{Theorem}[Camacho \cite{Camacho parabolic} pg. 84, Theorem 1] \label{difeoparab}
Let $f\colon (\mathbb{C},0) \to (\mathbb{C},0)$
be a germ of a holomorphic diffeomorphism tangent to the identity
$f(z) = z + \sum\limits_{j\ge 2} a_jz^j$, $a_2 \ne 0$. Then there
exist sectors $S^+$ and $S^-$ with vertices at $0 \in \mathbb{C}$, angles
$\pi-\te_0$ (where $0 < \te_0 < \pi/2)$ and opposite bisectrices in
such a way that:
\begin{itemize}
\item[{\rm(i)}] $f(S^+) \subset S^+$, $\lim\limits_{n\to+\infty} f^n(z)=0$, $\forall\, z \in S^+$
\item[{\rm(ii)}] $f^{-1}(S^{-}) \subset S^-$, $\lim\limits_{n\to+\infty} f^{-n}(z)=0$, $\forall\, z \in S^-$
\end{itemize}

\end{Theorem}
  
The two next lemmas deal with the hyperbolic and elliptic cases.

\begin{Remark} \label{difeohyperbolic}
Let $f(z) = \lambda z + O(z^{2})$, with $\mid \lambda \mid \neq 1$ be a hyperbolic germ of a diffeomorphism. Then there is a proper $f$-invariant sector (which is not a disc) only if $ 0< \lambda <1$.
\end{Remark}

\begin{proof}
Indeed, suppose that there is a sector $S = \{z \in \mathbb{C} : 0 < \mid z \mid < R, \alpha_{1} < arg(z) < \alpha_{2}\}$ such that $f(S) \subset S$.
It is known (\cite{bracci} pg. 14, Theorem 2.1 and Corollary 2.3) that there is a diffeomorphism $g$ such that $g \circ f \circ g^{-1}(z) = \lambda z$. Writing $z = re^{i\alpha} $ and $ \lambda = se^{i\beta}$ we have
\begin{eqnarray*}
g \circ f \circ g^{-1}(z) & = & \lambda z \\
& = & rse^{i(\alpha + \beta)}.
\end{eqnarray*}
Therefore, $f(S) \subset S$ only if  $\beta = 0$ and $\lambda = \mid \lambda \mid = s < 1$.
\end{proof}

\begin{Remark} \label{difeoeliptic}
Let $f(z) = \lambda z + O(z^{2})$ be a germ of difeomorphism with $\mid \lambda \mid = 1$, $\lambda ^{k} \neq 1$, for all $k \in \mathbb{N} \setminus \{0\}$. 
Then there are no proper $f$-invariant sectors (which is not a disc).
\end{Remark}

\begin{proof}
Let $S = \{z \in \mathbb{C} : 0 \leq \mid z \mid < R, \alpha_{1} < arg(z) < \alpha_{2}\}$ be a sector in $\mathbb{C}$. Consider the straight lines $r$ and $s$, that pass through the origin of $\mathbb{C}$ associated with the angles $\alpha_{1},\alpha_{2}$ respectively. That is, $r: y = tan(\alpha_{1})x$, and $s: y = tan(\alpha_{2})x$. Take the vectors $v_1 = e^{i\alpha_{1}} \in r$ and $v_2 =  e^{i\alpha_{2}}\in s$. We observe that $f'(0) =  e^{i\theta}$ for some $\theta \in \mathbb{R}$, $\theta \neq 0$ . Thus, 
\[
f'(0)\cdot v_1 =  e^{i\theta} e^{i\alpha_{1}} =  e^{i(\theta + \alpha_{1})}
\]
and
\[
f'(0)\cdot v_2 =  e^{i\theta} e^{i\alpha_{2}} =  e^{i(\theta + \alpha_{2})}.
\]
If $S$ is $f$-invariant then 
\[
\alpha_{1} \leq \theta + \alpha_{1} \leq \alpha_{2}
\]
and
\[
\alpha_{1} \leq \theta + \alpha_{2} \leq \alpha_{2}.
\]
This implies that $0 \leq \theta \leq \alpha_{2}-\alpha_{1}$ and $\alpha_{1} - \alpha_{2} \leq \theta \leq 0$ and, $\theta = 0$. Therefore, $\lambda = 1$, contradiction. 
\end{proof}

\section{Formal power series and change of coordinates}
In this section we prove a normal form for formal power series in one complex variable. We will denote by $\mathbb{C}[[z]]$ the ring of formal power series in the complex variable $z$. 
\begin{Lemma} \label{FormalChange}
Let $\hat{\varphi}(z) = \sum_{j = 1}^{\infty}a_{j}z^{j}$ be a formal power series in $\mathbb{C}[[z]]$. Then there are $\nu \in \mathbb{N}$ and an invertible formal power series $ \hat{\psi} \in \mathbb{C}[[z]]$ such that
\[
\hat{\varphi} \circ \hat{\psi}(z) = z^{\nu}
\]
for all $z \in \mathbb{C}$.
\end{Lemma}

For the proof of this lemma, we will need some results about formal power series.

%We start a version of the Inverse Function Theorem for power formal series. The references \cite{Carat, Henrici, Mark, Watson} have proofs of this result. The precise statement is 

\begin{Theorem} [\cite{cartanBook} pg. 15, Theorem 7.1] \label{teolagran}
Let $\hat{\varphi}(z) = \sum_{j = 0}^{\infty}a_{j}z^{j} \in \mathbb{C}[[z]]$. There exists $\hat{\psi} \in \mathbb{C}[[z]]$ such that  $ \hat{\varphi}(\hat{\psi}(u))= u $ if and only if $a_{0} = 0$ and $\hat{\varphi}'(0) = a_{1} \neq 0$. If such a $\hat{\psi}$ exists then it is unique and $\hat{\psi}(\hat{\varphi}(z)) = z$. 
\end{Theorem}

\begin{Theorem} [\cite{Ivan} pg. 874, Theorem 3]\label{RootsPowerSeries}
Let  $\hat{\varphi}(z) = \sum_{j = 0}^{\infty}a_{j}z^{j}$ be a formal power series, with $a_{0} = 1$ and let $n_{0}$ be any positive integer. Then there is a unique $\hat{\psi}(z) = \sum_{j = 0}^{\infty}b_{j}z^{j}$, with $b_{0} = 1$ such that $(\hat{\psi}(z))^{n_{0}} = \hat{\varphi}(z)$.
\end{Theorem}
\begin{Remark} \label{remarkRoots}
Theorem \ref{RootsPowerSeries} allows us to define $(\hat{\varphi}(z))^{\frac{1}{n_{0}}} = \hat{\psi}(z).$
\end{Remark}
\begin{Definition} [\cite{Ivan} pg. 878]{\rm \label{defLogExp}
Let $\hat{\varphi}(z) = \sum_{j=0}^{\infty}a_{j}z^{j}$ be a formal power series with $a_{0} = 1$ and $\hat{\mu}(z) = \sum_{j=1}^{\infty}a_{j}z^{j}$, where $\hat{\varphi} = 1 + \hat{\mu}$. The \textit{formal logarithm} of $\hat{\varphi}$ is
\[
L(\hat{\varphi}) = L(1 + \hat{\mu}) = \hat{\mu} - \frac{1}{2}\hat{\mu}^{2} + \frac{1}{4}\hat{\mu}^{4} + \dots = \sum_{j=1}^{\infty}(-1)^{j+1}\frac{\hat{\mu}^{j}}{j}.
\]
Given $\hat{\psi}(z) = \sum_{j=0}^{\infty}b_{j}z^{j}$ where $b_{0} = 0$, the \textit{formal exponential} of $\hat{\psi}$ is 
\[
E(\hat{\psi}) = 1 + \hat{\psi} + \frac{\hat{\psi}^{2}}{2!} + \frac{\hat{\psi}^{3}}{3!} + \dots = \sum_{n=0}^{\infty} \frac{\hat{\psi}^{n}}{n!}.
\]
}
\end{Definition} 
\begin{Theorem} [\cite{Ivan} pg. 880, Theorem 19]\label{Shift} 
Let $\hat{\varphi},\hat{\psi} \in \mathbb{C}[[z]]$ be formal power series  such that $\hat{\varphi}(z) = \sum_{j=0}^{\infty}a_{j}z^{j}$, with $a_{0} = 1$ and  $\hat{\psi}(z) = \sum_{j=1}^{\infty}b_{j}z^{j}$, with $b_{0} = 0$. Then $L(E(\hat{\psi})) = \hat{\psi}$ and $E(L(\hat{\varphi})) = \hat{\varphi}$.
\end{Theorem}
\begin{Remark} {\rm \label{remarkShift}
Theorem \ref{Shift} implies that there is a bijection between the formal power  series such that the constant term  is $1$ and the formal power series such the constant term is $0$.
}
\end{Remark}
\begin{proof}[\textbf{Proof of Lemma \ref{FormalChange}}]
Put $n_{0} = min \{j \in \mathbb{N}; a_{j} \neq 0\}$. Defining 
\[
\hat{\mu}(z) =\sum_{j \geq n_{0}+1}^{\infty} \frac{a_{j}}{a_{n_{0}}} z^{j - n_{0}}
\]
we can write 
\[
\hat{\varphi}(z) = a_{n_{0}}z^{n_{0}}[1 + \hat{\mu}(z)]
\]
with $\hat{\mu}(0) = 0$.
Let $b$ be a $n_{0}$-th root of $a_{n_{0}} \neq 0$. Define $
\hat{\gamma}(z) = bz(1 + \hat{\mu}(z))^{\frac{1}{n_{0}}}$. By Theorem \ref{RootsPowerSeries}, $(1 + \hat{\mu}(z))^{\frac{1}{n_{0}}}$ is a well defined formal power series (see Remark \ref{remarkRoots}) with constant term equals to $1$. Then, $\hat{\gamma}$  
is a well defined power formal series, such that 
\[
\hat{\gamma}(z)^{n_0} = a_{n_{0}}z^{n_{0}}(1 + \hat{\mu}(z))
\]
$\hat{\gamma}(0) = 0$ and $\hat{\gamma}'(0) = b \neq 0$. By Theorem \ref{teolagran} there is a formal power series $\hat{\psi}$, wich is an inverse for $\hat\gamma$. Therefore,
\begin{eqnarray*}
\hat{\varphi}\circ \hat{\psi}(z)&=& a_{n_{0}}(\hat\psi(z))^{n_0}[1 + \hat{\mu}(\hat{\psi}(z))] \\
 &=& (\hat{\gamma} \circ \hat{\psi} (z))^{n_0} \\
 &=& z^{n_{0}}.
\end{eqnarray*}

\end{proof}

\section{Invariance and finiteness} \label{FinitSubgroup}
This section is dedicated to the study of subgroups $G$ of $Diff(\mathbb{C},0)$ admitting an invariant sector $S$ and a suitable holomorphic function $\varphi$ in $S$ which is constant in the pseudorbits of $G$ in $S$. 

\begin{Lemma} \label{FormalInvar}
Let $\varphi:S \rightarrow \mathbb{C}$ be a holomorphic function that admits a nonzero asymptotic expansion $\hat{\varphi} (z)$ in a sector $S \subset \mathbb{C}$. Suppose that there is a holomorphic function $f:U \rightarrow \mathbb{C}$ defined in a neighborhood $U$ of $0 \in \mathbb{C}$ with $f(0)=0$ and $f(S\cap U)\subset S\cap U$. If $\varphi \circ f = \varphi$ in $S\cap U$ then $\hat{\varphi} \circ f = \hat{\varphi}$ in $S\cap U$.
\end{Lemma}

\begin{proof}
For simplicity of notation we may suppose $S\subset U$. We fix $k \in \mathbb{N}$. Using the asymptotic expansion of $\varphi$ in $S$ there are $B_{k} \in \mathbb{C}$ such that for each $z \in S$
\begin{eqnarray*}
\mid \varphi \circ f(z) - \sum_{j=0}^{k-1} a_{j}[f(z)]^{j} \mid &\leq& B_{k}z^{k}.
\end{eqnarray*}
Since $\hat{\varphi} \circ f(z) =\sum_{j=0}^{\infty} a_{j}[f(z)]^{j}  $ and $\varphi \circ f = \varphi$, so for all $z \in S$, we have

\begin{eqnarray*}
\mid \varphi(z) - \sum_{j=0}^{k-1} a_{j}[f(z)]^{j} \mid &=& \mid \varphi(z) - \sum_{j=0}^{k-1} a_{j}z^{j} + \sum_{j=0}^{k-1} a_{j}z^{j} - \sum_{j=0}^{k-1} a_{j}[f(z)]^{j} \mid \\
&\leq& \mid \varphi(z) -\sum_{j=0}^{k-1} a_{j}z^{j}\mid + \mid \sum_{j=0}^{k-1} a_{j}z^{j} - \sum_{j=0}^{k-1} a_{j}[f(z)]^{j} \mid \\
&=& \mid \varphi(z) -\sum_{j=0}^{k-1} a_{j}z^{j}\mid + \mid \sum_{j=0}^{k-1} a_{j}z^{j} - \varphi(z)+ \varphi(z) - \sum_{j=0}^{k-1} a_{j}[f(z)]^{j} \mid \\
&\leq& \mid \varphi(z) -\sum_{j=0}^{k-1} a_{j}z^{j}\mid + \mid \sum_{j=0}^{k-1} a_{j}z^{j} - \varphi(z)\mid + \mid \varphi(z) - \sum_{j=0}^{k-1} a_{j}[f(z)]^{j} \mid \\
&=& \mid \varphi(z) -\sum_{j=0}^{k-1} a_{j}z^{j}\mid + \mid \varphi(z) - \sum_{j=0}^{k-1} a_{j}z^{j}\mid + \mid \varphi \circ f(z) - \sum_{j=0}^{k-1} a_{j}[f(z)]^{j} \mid \\
&\leq& A_{k}z^{k} + A_{k}z^{k} + B_{k}z^{k} \\
&=& [2A_{k} + B_{k}]z^{k}.
\end{eqnarray*}
We conclude that there is a constant $C_{k} = 2A_{k} + B_{k}$ such that $\mid \varphi(z) - \sum_{j=0}^{k-1} a_{j}[f(z)]^{j} \mid \leq C_{k} z^{k}$, $\forall z \in S$, i.e., $\varphi$ admits $\sum_{j=0}^{\infty} a_{j}[f]^{j} = \hat{\varphi} \circ f $ as asymptotic expansion in $S$. Therefore, by unicity of the asymptotic expansion (see Theorem \ref{BusAsymp}),
\[
\hat{\varphi} \circ f(z) = \hat{\varphi}(z), \forall z \in S.
\]
\end{proof}

Next we stablish a key point in our argumentation concerning groups of germs of diffeomorphisms. Let $\varphi: S \rightarrow \mathbb{C}$ be a non-constant holomorphic function that admits a nonzero asymptotic expansion $\hat{\varphi} (z)$ in a sector $S \subset \mathbb{C}$. We define the \textit{invariance group} of $\varphi$ as 
\[
Inv_{S}(\varphi) = \left\lbrace f \in Diff(\mathbb{C},0) : f(S) \subset S  \textit{ and }\varphi\circ f = \varphi \textit{ in } S \right\rbrace.
\]
This is a subgroup of $Diff(\mathbb{C},0)$.
\begin{Example} \rm{
Consider the function $\varphi(z) = cos\left(\frac{2 \pi}{z}\right)$. Then $Inv(\varphi) = <f,g>$, where $f(z) = -z$ and $g(z) = \frac{1}{1-z}$. The invariance group $Inv(\varphi)$ is not finite. The point is that $\varphi(z)$ does not admit an asymptotic expansion. Let us prove this. For every $z \neq 0$, 
\begin{eqnarray*}
\varphi'(z) &=& \frac{2\pi}{z^{2}}sin\left(\frac{2\pi}{z}\right).
\end{eqnarray*} 
Consider the sector $S = \{z \in \mathbb{C} : 0 \leq \mid z \mid < R, - \alpha_{0} <  arg(z) < \alpha_{0}\}$. For every $0 < \alpha < \alpha_{0}$, let $S^{*} \subset S$ be a proper subsector $S^{*} = \{z \in \mathbb{C} : 0 \leq \mid z \mid \leq r < R, - \alpha \leq  arg(z) \leq \alpha\}$.  
Let $(z_{n})_{n \in \mathbb{N}}$ be a sequence of complex numbers given by $z_{n} = \frac{r}{n}( 1+ i \tan \alpha)$. Then
\begin{eqnarray} \label{eq1Ex}
\lim_{n \rightarrow \infty}\varphi'(z_{n}) = {2\pi}\lim_{n \rightarrow \infty} \frac{n^{2}}{r^{2} ( 1+ i \tan \alpha )^{2}} \sin\left(\frac{2n\pi}{r( 1+ i \tan \alpha)}\right).
\end{eqnarray}
We observe that
\begin{eqnarray*}
\frac{2n\pi}{r( 1+ i \tan \alpha)} & = & \frac{2n\pi}{r( 1+ i \tan \alpha)} \frac{( 1 - i \tan \alpha)}{( 1 - i \tan \alpha)} \\
& = & \frac{2n\pi ( 1 - i \tan \alpha)}{r( 1 + \tan^{2} \alpha)} \\
& = & \frac{2n\pi}{r( 1 + \tan^{2} \alpha)} - i \frac{2n\pi \tan \alpha}{r( 1 + \tan^{2} \alpha)}.
\end{eqnarray*}
Furthermore,
\begin{eqnarray*}
i \frac{2n\pi}{r( 1+ i \tan \alpha)} & = & i\frac{2n\pi}{r( 1 + \tan^{2} \alpha)} +  \frac{2n\pi \tan \alpha}{r( 1 + \tan^{2} \alpha)}
\end{eqnarray*}
and,
\begin{eqnarray*}
- i \frac{2n\pi}{r( 1+ i \tan \alpha)} & = &  - i\frac{2n\pi}{r( 1 + \tan^{2} \alpha)} - \frac{2n\pi \tan \alpha}{r( 1 + \tan^{2} \alpha)}.
\end{eqnarray*}
Then,
\begin{eqnarray*}
\sin\left(\frac{2n\pi}{r( 1+ i \tan \alpha)}\right) & = & e^{i \left(\frac{2n\pi}{r( 1+ i \tan \alpha)}\right) } - e^{-i \left(\frac{2n\pi}{r( 1+ i \tan \alpha)}\right) }\\
& = &  e^{i\frac{2n\pi}{r( 1 + \tan^{2} \alpha)}} e^{\frac{2n\pi \tan \alpha}{r( 1 + \tan^{2} \alpha)}} - e^{- i\frac{2n\pi}{r( 1 + \tan^{2} \alpha)} } e^{- \frac{2n\pi \tan \alpha}{r( 1 + \tan^{2} \alpha)}}.
\end{eqnarray*}
From this last equation and equation \ref{eq1Ex} we have
\begin{eqnarray*}
\lim_{n \rightarrow \infty}\varphi'(z_{n}) &=& {2\pi}\lim_{n \rightarrow \infty} \frac{n^{2}}{r^{2} ( 1+ i \tan \alpha )^{2}} \sin\left(\frac{2n\pi}{r( 1+ i \tan \alpha)}\right) \\
& = & {2\pi}\lim_{n \rightarrow \infty} \frac{n^{2}}{r^{2} ( 1+ i \tan \alpha )^{2}} [e^{i\frac{2n\pi}{r( 1 + \tan^{2} \alpha)}} e^{\frac{2n\pi \tan \alpha}{r( 1 + \tan^{2} \alpha)}} - e^{- i\frac{2n\pi}{r( 1 + \tan^{2} \alpha)} } e^{- \frac{2n\pi \tan \alpha}{r( 1 + \tan^{2} \alpha)}}] \\
& = & + \infty.
\end{eqnarray*}
On the other hand, consider the sequence $w_{n} = \frac{r}{n}( 1+ i \tan( -\alpha))$. We know that $\tan( -\alpha) = - \tan \alpha$. Therefore,
\begin{eqnarray*}
\sin\left(\frac{2n\pi}{r( 1 - i \tan \alpha)}\right) & = &  e^{i\frac{2n\pi}{r( 1 + \tan^{2} \alpha)}} e^{\frac{- 2n\pi \tan \alpha}{r( 1 + \tan^{2} \alpha)}} - e^{- i\frac{2n\pi}{r( 1 + \tan^{2} \alpha)} } e^{+ \frac{2n\pi \tan \alpha}{r( 1 + \tan^{2} \alpha)}}
\end{eqnarray*}
and
\begin{eqnarray*}
\lim_{n \rightarrow \infty}\varphi'(w_{n}) & = & {2\pi}\lim_{n \rightarrow \infty} \frac{n^{2}}{r^{2} ( 1+ i \tan \alpha )^{2}} [e^{i\frac{2n\pi}{r( 1 + \tan^{2} \alpha)}} e^{\frac{- 2n\pi \tan \alpha}{r( 1 + \tan^{2} \alpha)}} - e^{- i\frac{2n\pi}{r( 1 + \tan^{2} \alpha)} } e^{+ \frac{2n\pi \tan \alpha}{r( 1 + \tan^{2} \alpha)}}] \\
& = & - \infty.
\end{eqnarray*}
This implies that $\lim_{z \rightarrow 0}\varphi'(z)$ does not exist. By Theorems \ref{coefasymp2} and \ref{BusAsymp}, $\varphi$ does not admit an asymptotic expansion.  
}
\end{Example}

\begin{Proposition} \label{INVfinite}
Let $\varphi: S \rightarrow \mathbb{C}$ be a non-constant holomorphic function that admits a nonzero asymptotic expansion $\hat{\varphi} (z)$ in a sector $S \subset \mathbb{C}$. Then 
$
Inv_{S}(\varphi) 
$
is a finite subgroup of $Diff(\mathbb{C},0)$.
\end{Proposition}

\begin{proof}
By Lemma \ref{FormalInvar} $f \in Inv_{S}(\varphi)$ implies that $f \in Inv_{S}(\hat{\varphi})$ where 
\[
Inv(\hat{\varphi}) = \left\lbrace g \in Diff(\mathbb{C},0) : \hat{\varphi}\circ g = \hat{\varphi} \textit{ (as power series)}\right\rbrace.
\]
Therefore, $\sharp Inv_{S}(\varphi) \leq \sharp Inv(\hat{\varphi})$. It is enough to show that $Inv(\hat{\varphi})$ is finite. By Lemma \ref{FormalChange} there are a number $n_{0} \in \mathbb{N}$ and a formal power series $\hat{\psi}$ such that $\hat{\varphi} \circ \hat{\psi}(z)= z^{n_{0}}$ for every $z$ in $S$.
\begin{Claim}
There is an isomorphism (bijection) 
\begin{eqnarray*}
\mathcal{G}: Inv(\hat{\varphi}) &\rightarrow & Inv(z^{n_0}) \\
 g &\mapsto & \hat{\psi}^{-1} \circ g  \circ \hat{\psi}
\end{eqnarray*}
\end{Claim}
In fact, suppose that $g \in Inv(\hat{\varphi})$ then $\hat{\varphi} \circ g = \hat{\varphi}$ implies that
\begin{eqnarray*}
\hat{\varphi} \circ \hat{\psi} \circ \hat{\psi}^{-1} \circ g & = & \hat{\varphi} \\
z^{n_0} \circ \hat{\psi}^{-1} \circ g & = & \hat{\varphi} \\
z^{n_0} \circ \hat{\psi}^{-1} \circ g \circ \hat{\psi}& = & \hat{\varphi} \circ \hat{\psi} \\
z^{n_0} \circ \hat{\psi}^{-1} \circ g \circ \hat{\psi}& = & z^{n_0}.
\end{eqnarray*}
It remains then to observe that $Inv(z^{n_0})$ is the finite cyclic group generated by the rotation $\theta(z) = e^{\frac{2\pi i}{n_{0}}}z$.
\end{proof}

\section{Foliations and reduction of singularities} \label{ReviewFoliations}
Now we will provide a brief review of some concepts on holomorphic foliations with singularities. For further references see \cite{Ca-Sad-LN generalizedCurve, Ca-Sa-LN algebraic limit, seidenberg}.

Let $\fa$ be a holomorphic foliation with isolated singularities on a complex surface $M$. Denote by
sing$(\fa)$ the singular set of $\fa$. Given a leaf $L_0$ of $\fa$ we choose any base point $p \in L_0 \in M \setminus sing(\fa)$ and a transverse disc $\Sigma_p \Subset M$ to $\fa$ centered at $p$. The \textit{holonomy group} of the leaf $L_0$ with respect to the disc $\Sigma_p$ and to the base point $p$ is the image of the representation $\textit{Hol}: \pi_{1}(L_0, p) \rightarrow Diff(\Sigma_p, p)$ obtained by lifting closed paths in $L_0$ with base point $p$, to paths in the leaves of $\fa$, starting at points $z \in \Sigma_p$, by means of a transverse fibration to $\fa$ containing the disc $\Sigma_p$ (see \cite{CamachoGeometTheory}, chapter IV). Given a point $z \in \Sigma_p$ we denote the leaf through $z$ by $L_z$. Given a closed path $\gamma \in \pi_{1}(L_{0}, p)$ we denote by $\tilde{\gamma}$ its lift to the leaf $L_z$ and starting from the point $z$. Then the image of the corresponding holonomy map is $h_{[\gamma]}(z) = \tilde{\gamma}_{z}(1)$, i.e., the final point of the lifted path $\tilde{\gamma}_{z}$. This defines a diffeomorphism germ map $h_{[\gamma]}: (\Sigma_{p}, p) \rightarrow (\Sigma_{p}, p)$ and also a group homomorphism $\textit{Hol}: \pi_{1}(L_{0}, p) \rightarrow Diff(\Sigma_{p}, p)$. The image $\textit{Hol}(\fa, L_{0}, \Sigma_p, p) \subset Diff(\Sigma_{p}, p)$ of such homomorphism is called the \textit{holonomy group} of the leaf $L_0$ with respect to $\Sigma_{p}$ and $p$. By considering any parametrization $z : (\Sigma_{p}, p) \rightarrow (\mathbb{D},0)$ we may identify (in a non-canonical way) the holonomy group with a subgroup of $Diff(C,0)$. It is clear from the construction that the maps in the holonomy group are constant in the traces of each leaf of the foliation in the given transverse section. Nevertheless, this property can be
shared by a larger group that may therefore contain more information about the foliation in a
neighborhood of the leaf. The \textit{virtual holonomy group} of the leaf with respect to the transverse
section $\Sigma_{p}$ and base point p is defined as (\cite{Ca-Sa-LN algebraic limit} Definition 2, page 432 or also \cite{algebraicLimit})
\[
\textit{Hol}^{virt}(\fa, \Sigma_{p}, p) = \{f \in Diff(\Sigma_{p}, p): \tilde{L}_z = \tilde{L}_{f(z)}, \forall z  \in (\Sigma_{p}, p)\}.
\]
The virtual holonomy group contains the holonomy group and consists of the map germs that are constant in the trace of each leaf of the foliation in a transverse section.

\textbf{Reduction of singularities} \label{subsectionReduction}
We denote by $\fa$ a representative of a germ of holomorphic foliation with a singularity at the origin $0\in\mathbb{C}^2$, defined in an open neighborhood $U$ of the origin such that $0$ is the only singularity of $\fa$ in $U$. Theorem of redution of singularities of Seindenberg \cite{seidenberg} asserts the existence of a proper holomorphic map $\sigma:\tilde{U}\to U$ which is a finite composition of quadratic blowing up's, starting with a blowing up at the origin, such that the pull-back foliation $\tilde{\fa}:=\sigma^*(\fa)$ of $\fa$ by $\sigma$ satisfies:
\begin{enumerate}
\item The exceptional divisor $E(\fa):=\sigma^{-1}(0)\subset \tilde{U}$  can be written as $E(\fa)=\bigcup_{j=1}^{m} D_j$ where each irreducible component $D_j$ is diffeomorphic to an embedded projective line $\mathbb{P}(1)$ introduced as a divisor of the successive blowing up's. 
\item $\sing(\fa)\subset E(\fa)$ is a finite set and any singularity $p\in\sing(\tilde{\fa})$ is {\it irreducible} i.e., belongs to one of the following categories:
\begin{itemize}
\item  $xdy - \lambda ydx + h.o.t. = 0$, $\lambda\notin\bq_{+}$ ({\it non-degenerate singularity}),
\item  $y^{p+1}dx-[x(1+\lambda y^p)+h.o.t.] dy = 0$, $p\geq 1$. This case is called a {\it saddle-node}.
\end{itemize}
\end{enumerate}

We call the lifted foliation $\tilde{\fa}$ the {\it desingularization} or {\it reduction of singularities} of $\fa$. The foliation is {\it non-dicritical} if $E(\fa)$ is invariant by $\tilde{\fa}$. This is equivalent to say that $\fa$ admits only a finite number of separatrices. Any two components $D_i$ and $D_j$, $i\neq j$, of the exceptional divisor, intersect (transversely) at at most one point, which is called a {\it corner}. There are no cycles and no triple points in the exceptional divisor.

The following is an important result relates analytic conjugation for germs of singularities and for their holonomy maps.

\begin{Theorem} [Martinet-Ramis \cite{Ramis1} pg. 595]
\label{Theorem:analyticconjugation0} Let $\fa_1$\,, $\fa_2$ be two
germs of non-degenerate singularities  $\fa_j\colon xdy-\la
y(1+b_j(x,y))dx=0$, with $b_j(x,y)$ holomorphic, $b_j(0,0)=0$, $\la
\in \mathbb{R}_{-}$. Denote by $f_j\colon \mathbb{C},0 \to \mathbb{C},0$ the holonomy map
of $\Ga\colon (y=0)$ with respect to $\fa_j$\,. Then $\fa_1$ and
$\fa_2$ are analytically conjugate by a holomorphic diffeomorphism
$\Phi\colon \mathbb{C}^2,0 \to \mathbb{C}^2,0$ if, and only if, the holonomy maps
$f_1$ and $f_2$ are analytically conjugate in $\Diff(\mathbb{C},0)$.
\end{Theorem}

As a corollary of Theorem \ref{Theorem:analyticconjugation0} above we have:

\begin{Theorem} [Martinet-Ramis \cite{Ramis1} pg. 595]\label{Theorem:analyticconjugation}
{\it A germ of an irreducible non-degenerate singularity $\fa\colon xdy-\la ydx +\cdots = c$
is analytically linearizable if, and only if, its holonomy map of a given separatrix is analytically linearizable.}
\end{Theorem}

\section{Examples} \label{ExampleSection}
This section is dedicated to show the necessity of the hypotheses in our main results. The first example justifies the non-dichricity hypothesis of Theorem \ref{Main}. 
\begin{Example} {\rm
Consider $\fa$ a holomorphic foliation with an isolated singularity at $0 \in \mathbb{C}^2$ and the pull-back foliation $\widetilde{\fa} := \pi^*(\fa)$, where
$\pi = \pi_{2} \circ \pi_{1} \colon M_{2} \to  \mathbb{C}^2$ is the blow-up resolution of $\fa$. Suppose that 
$\widetilde{\fa}$ exhibits a dicritical component $\bp_{2}$, where the exceptional divisor $D = \pi^{-1}(0)$ is given by $D = \bigcup\limits_{j=1}^{2} \bp_j$. Then, the component $\bp_{2}$ is transverse to $\widetilde{\fa}$ (without tangent points).
(See figure \ref{figura:DicriticExample} below).
\begin{figure}[h]
\centering % para centralizarmos a figura
\includegraphics[width=10cm]{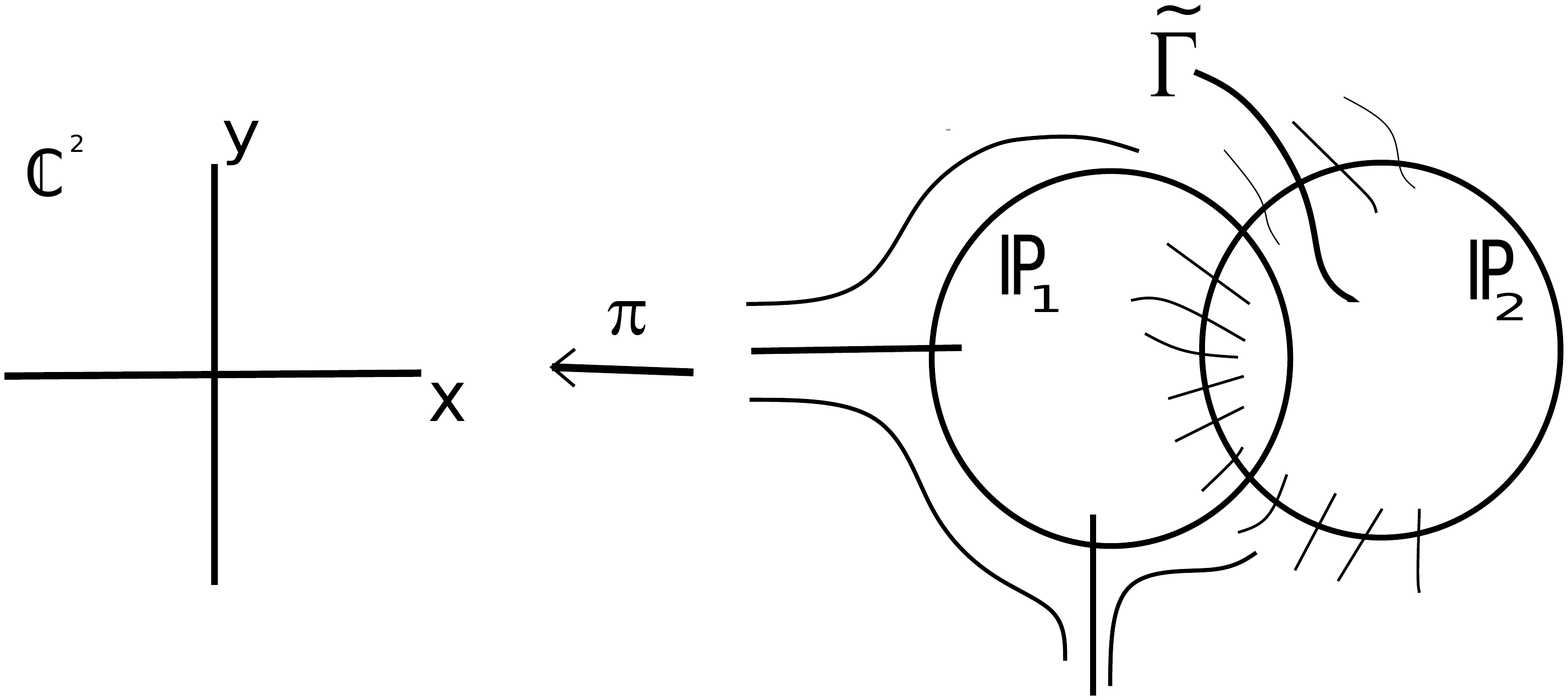} % leia abaixo
\caption{$\mathbb{P}_2$ is a dicritical component of $D$.}
\label{figura:DicriticExample}
\end{figure}
Consider a separatrix $\widetilde{\Gamma}$ transverse to $\mathbb{P}_2$ and $\widetilde{\Sigma}\subset \mathbb{P}_2$ a transverse section such that $\widetilde{\Sigma}$ $ \overline{\pitchfork} $ $\widetilde{\Gamma}$. Since $\mathbb{P}_2$ is a dicritic component there is a neighborhood $\widetilde{U}$ of $\widetilde{\Sigma}$ such that  for $\widetilde{\fa}|_{\widetilde{U}}$ admits a holomorphic first integral. Denote $U = \pi(\widetilde{U})$, $\Gamma = \pi(\widetilde{\Gamma})$ and $\Sigma = \pi(\widetilde{\Sigma})$. Then there is a holomorphic first integral for $\fa$ defined in the neighborhood $U$ of $\Sigma$ in $\mathbb{C}^2$. The restriction of this holomorphic first integral to $\Sigma$ implies that the pair $(\fa,\Gamma)$ admits a moderate sectorial first integral. On the other hand, we maiy obtain such a foliation $\fa$ for which the holonomy group of the component $\mathbb{P}_1$ is not finite (see \cite{LN1}). Therefore, $\fa$ does not admit a holomorphic first integral defined in a neighborhood of $0 \in \mathbb{C}^2$. The point is that $\fa$ is dicritic. 
}
\end{Example}

In the next example we show the necessity of the hypothesis about non-zero asymptotic expansion in main results. This example will also motivate the notion of \textit{non-weak separatrix} that we will do in the section §\ref{null asymptotic saddle-node}.

\begin{Example} {\rm \label{ModerateNecessity}
Consider the saddle-node $\fa$ given by
$$ \omega_{2,1} (x,y)=
\begin{cases}
\dot{x} = x \\
\dot{y} = y^{2}. \\
\end{cases}
$$  
Observe that 
\[
xdy - y^{2}dx =0 \Rightarrow \frac{dy}{y^{2}} -\frac{dx}{x} =0 \Rightarrow d(-\frac{1}{y} - lnx) = 0. 
\]
Consider the strong separatrix $\Gamma_{s} : \{y=0\}$, the weak separatrix $\Gamma_{w} : \{x = 0\}$ and the transverse sections $\Sigma_{s} : (x = 1)$ and $\Sigma_{w} : (y = 1)$.  Then the foliation admits the first integral $f(x,y) = xe^{\frac{1}{y}}$ which is holomorphic outside of the $\Gamma_{s}$ .

\begin{figure}[h]
\centering % para centralizarmos a figura
\includegraphics[width=8cm]{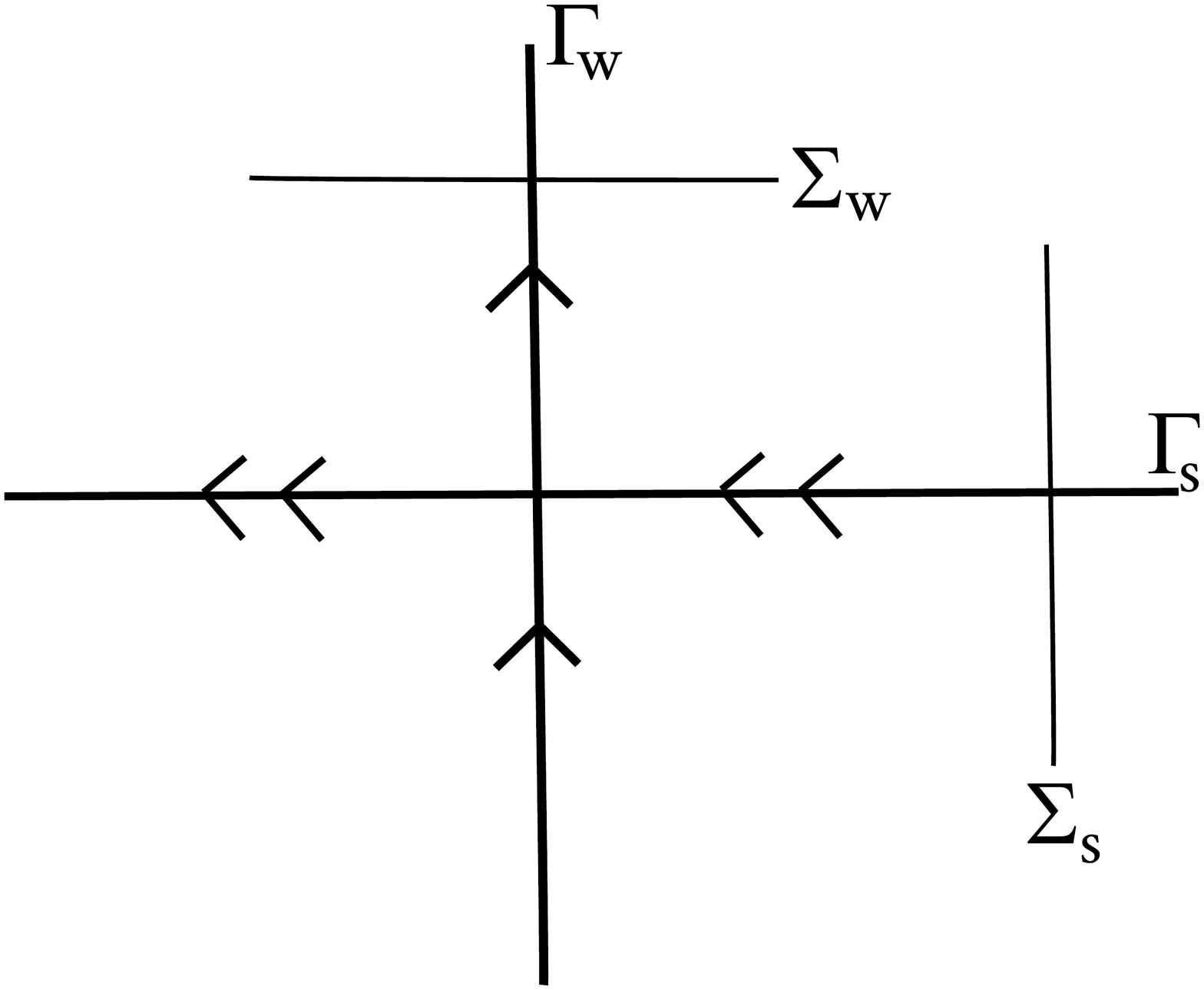} % leia abaixo
\caption{Saddle-node example: Transversal sections and separatrices.}
\label{figura:transversews}
\end{figure}
Denote by $\Sigma_{w}^{*} := \Sigma_{w} \setminus \{(0,1) \in \mathbb{C}^2\}$ and $\Sigma_{s}^{*} := \Sigma_{s} \setminus \{(1,0) \in \mathbb{C}^2\}$.  The restriction $\varphi_{w} = f|_{S_w} = xe$ is a sectorial first integral and admits nonzero asymptotic expansion at any sector $S_{w} \subset \Sigma_{w}^{*}$. Consider a sector $S_{s} \subset \{y \in \Sigma_{s}^{*} : Re(y) < 0\}$. Since $\varphi_{s}(y) = e^{\frac{1}{y}}$ we have
\begin{eqnarray*}
\mid \varphi_{s}(y) - 0 \mid &=& \mid e^{\frac{1}{y}} \mid = e^{Re\left(\frac{1}{y}\right)} = e^{Re\left(\frac{\bar{y}}{\mid y\mid^{2}}\right)}\\
&=& e^{\frac{1}{\mid y\mid^{2}}Re\left(\bar{y}\right)} = e^{\frac{1}{\mid y\mid^{2}}Re\left({y}\right)}.
\end{eqnarray*}
Since $e^{\frac{1}{\mid y\mid^{2}}Re\left({y}\right)} \in \mathbb{R}$ and $Re(y)<0$ for all $k \in \mathbb{N}$ there is a constant $A_k >0$ such that  
\begin{eqnarray*} \label{asympf1}
e^{\frac{1}{\mid y\mid^{2}}Re\left({y}\right)} &\leq & A_{k} \mid y \mid^k
\end{eqnarray*}
$\forall y$ in a proper subsector $S'\subset S_{s}$. Therefore, $\varphi_{s}$ admits $0 \in \mathbb{C}$ as asymptotic expansion in $S_{s}$. Nevertheless, $\fa$ does not admit a holomorphic first integral. The point is that:
\\
\item[i)]  $\varphi_{w}$ admits a nonzero asymptotic expansion, but $\Gamma_{w}$ is a a weak separatrix;
\\
\item[ii)] $\varphi_{s}$ admits zero as asymptotic expansion.
}
\end{Example}

\section{Dulac correspondence and transport by holonomy} \label{DulacSection}
In this section we will describe an important tool used to prove Theorem \ref{Main}. This is the Dulac correspondence (see  \cite{algebraicLimit, Cerveau-Scardua99, OrbBounded, integrable}).

Motivated by Seindenberg reduction theorem (see section \ref{ReviewFoliations}), we consider $\widetilde{\mathcal{F}}$ a foliation on a compact complex surface $\tilde{M}$ and $D \subset \tilde{M}$ a compact (codimension one) invariant divisor with normal crossing, no cycles and no triple points. We write $D = \cup_{j=1}^{m}D_{j}$, where each $D_{j}$ is an irreducible smooth component, and fix local transverse sections $\Sigma_{j}$ such that $\Sigma_{j} \cap D_{j} = p_{j} \notin sing(\widetilde{\mathcal{F}})$, and $(\Sigma_{j}, p_{j}) \cong (\mathbb{C},0)$. 
Denote by $G_j$ the holonomy group $Hol(\tilde{\mathcal{F}},D_j,\Sigma_j)$ of $D_j$ (we refer to section \ref{ReviewFoliations} for definition of holonomy group or \cite{mattei-moussu} for more properties). Denote by $L_z$ the leaf of $\tilde{\mathcal{F}}$ that contains the point $z \in \tilde{M}$ . 
The \textit{virtual holonomy group} $\widehat{G}_j$ of $\tilde{M}$ relative to the component $D_j$ at the section $\Sigma_j$ is defined to be (\cite{Ca-Sa-LN algebraic limit})
\[
\widehat{G}_j = \hat{Hol}(\tilde{\mathcal{F}}, D_{j}, \Sigma_{j}) = \{f \in Diff(\Sigma_j,p_j)\ \tilde{L}_{z} = \tilde{L}_{f(z)} \textit{, for any } z \in (\Sigma_{j}, p_{j})\}.
\]
Clearly, this virtual holonomy group $\widehat{G}_j$ contains the holonomy group $G_j$. Now we will fix a corner $q = D_i \cap D_j$  and make the following assumption: \\
\\
\\
ASSUMPTION . The corner $q$ is an irreducible singularity with a holomorphic first integral.
\\
\\
\begin{figure}[h]
\centering % para centralizarmos a figura
\includegraphics[width=10cm]{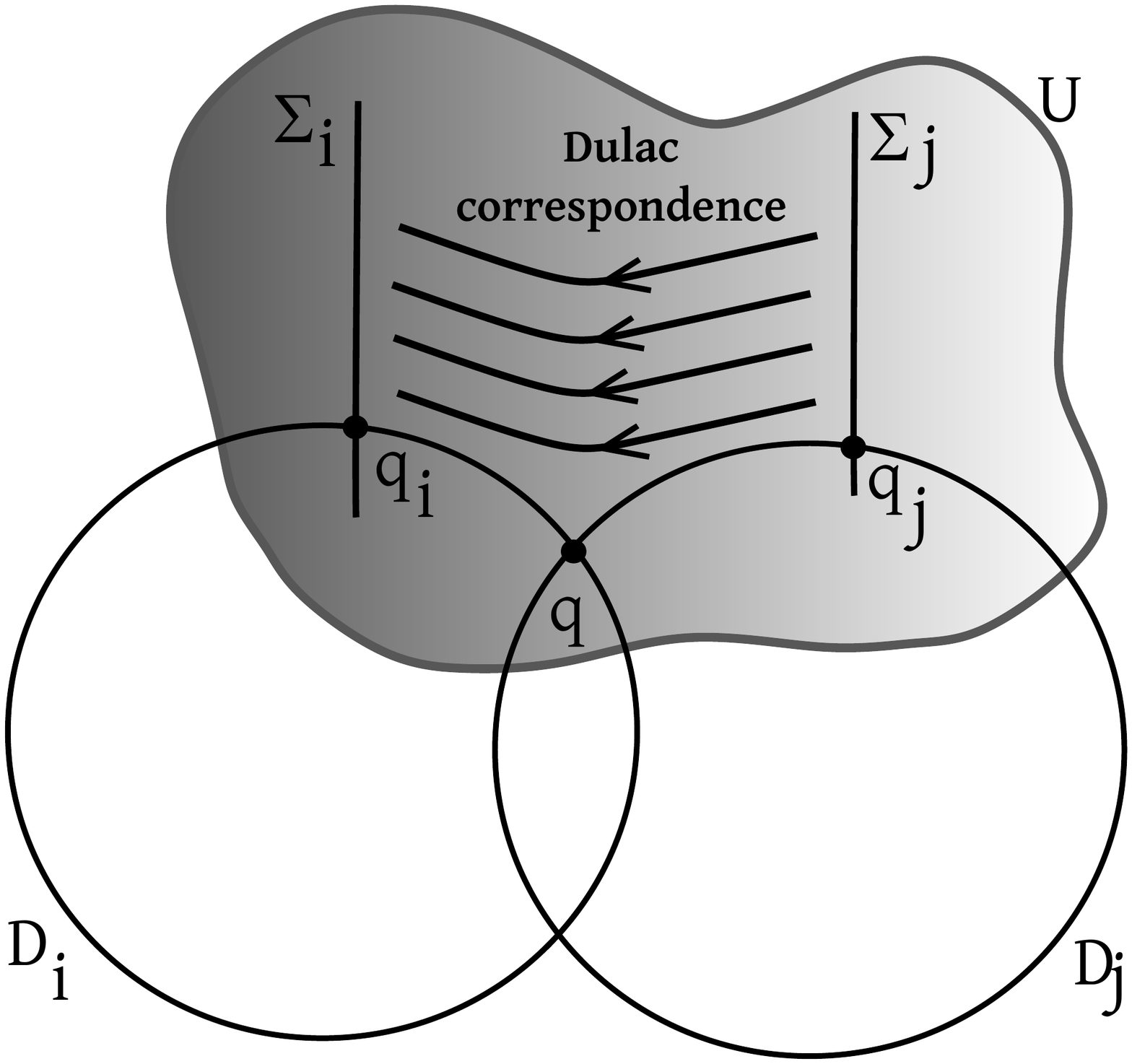} % leia abaixo
\caption{Dulac correspondence.}
\label{figura:DulacIntro}
\end{figure}

Since $q$ is irreducible with a holomorphic first integral, by Mattei-Moussu theorem (Theorem \ref{Theorem:analyticconjugation}) there are local holomorphic coordinates $(x, y) \in U$ such that $D_i \cap U = \{x = 0\}, D_j \cap U = \{y = 0\}$, and such that $\tilde{\mathcal{F}}|_{U}$ is given in the \textit{normal form} as $nxdy + mydx = 0$ and $q : x = y = 0$, where $n/m \in \mathbb{Q}_{+}$ and $\langle n, m\rangle = 1$. We fix the local transverse sections
as $\Sigma_{j} = \{x = 1\}$ and $\Sigma_{i} = \{y = 1\}$, such that $\Sigma_i \cap D_i = q_i \neq q$ and $\Sigma_j \cap D_j = q_j \neq q$. Denote by $h_{0} \in \textit{Hol}(\tilde{\mathcal{F}},D_i, \Sigma_i)$ the element corresponding to the corner $q$. Then we have $h_{0}(x) = exp(-2\sqrt{-1}\pi n/m)x$. The local leaves are given by $x^{m}y^{n} =cte$. The \textit{Dulac correspondence} is therefore given by (all branches are considered)
\[ \label{EqDulac}
\mathcal{D}_q : \mathcal{F}(\Sigma_{i}) \rightarrow \mathcal{F}(\Sigma_{j}),\textit{  }\mathcal{D}_q(x) = \{x^{m/n}\}.
\]
For the purpose of this paper we may assume that $G_{i}$ is abelian (for the another cases see  \cite{algebraicLimit, Cerveau-Scardua99, OrbBounded, integrable}). Take any element $h \in G_i$. Since $G_i$ is abelian, we have $h(x) = \mu x \tilde{h}(x^{m})$ for some $\tilde{h} \in \mathcal{O}_{1}, \tilde{h}(0) = 1$. We take $\mu_{1} = \mu^{m/n}$ and $h_{1} = \tilde{h}^{m/n}$ to be one of the $n$-roots of $\mu^{m}$ and $\tilde{h}^{m}$, respectively. Then we define $h^{\mathcal{D}_{q}} : (\Sigma_j , q_j) \rightarrow (\Sigma_j, q_j)$ by $h^{\mathcal{D}_q}(y) = \mu_{1}y\tilde{h}_1(y^{n})$. We consider the collection $\{h^{\mathcal{D}}\}$ of all these elements.

We will use the Dulac correspondence inside the proof of Theorem \ref{Main}, to transport the moderate sectorial first integral between two different projective lines with a common corner. More precisely, we have the following lemma.

\begin{Lemma} \label{LemmaDulac}
Let $\widetilde{\mathcal{F}}$ be a foliation on a compact complex surface $\tilde{M}$, and  $D \subset \tilde{M}$ a compact (codimension one) invariant divisor with normal crossing and no triple points. We write $D = \cup_{j=1}^{m}D_{j}$, where each $D_{j}$ is an irreducible smooth component. Let $D_1, D_2$ be two components such that the corner $q \in D_1 \cap D_2$ admits a neighborhood $U$ of $q$, where $\widetilde{\mathcal{F}}$ can be written in the normal form $nxdy + mydx = 0$. Suppose that there is a separatrix $\Gamma_1 \subset D_{1}$ such that the pair $(\widetilde{\mathcal{F}}, \Gamma_1)$ admits a moderate sectorial first integral $\varphi_1 :S_1 \rightarrow \mathbb{C}$, where $S_1$ is a sector contained in a transverse section $\Sigma_1$ and vertex at a point $q_1 \in \Sigma_1$.  Then for every point $q_{2} \in D_2 \cap U \setminus \{q\}$, there are a  transverse section $\Sigma_{2} \ni q_{2}$, a sector $S_2 \subset \Sigma_{2}$ with vertex at $q_{2}$, and a holomorphic function  $\varphi_{2} : S_2 \rightarrow \mathbb{C}$, such that $\varphi_{2}$ is constant in the traces of each leaf of $\fa$ in ${S_2}$ and which admits a nonzero asymptotic expansion. Moreover if $\D_{q} :\fa(\Sigma_1) \rightarrow \fa(\Sigma_2)$ denotes the Dulac correspondence then $\varphi_1 \circ \D_q = \varphi_2 $.
\end{Lemma}
\begin{figure}[h]
\centering % para centralizarmos a figura
\includegraphics[width=10cm]{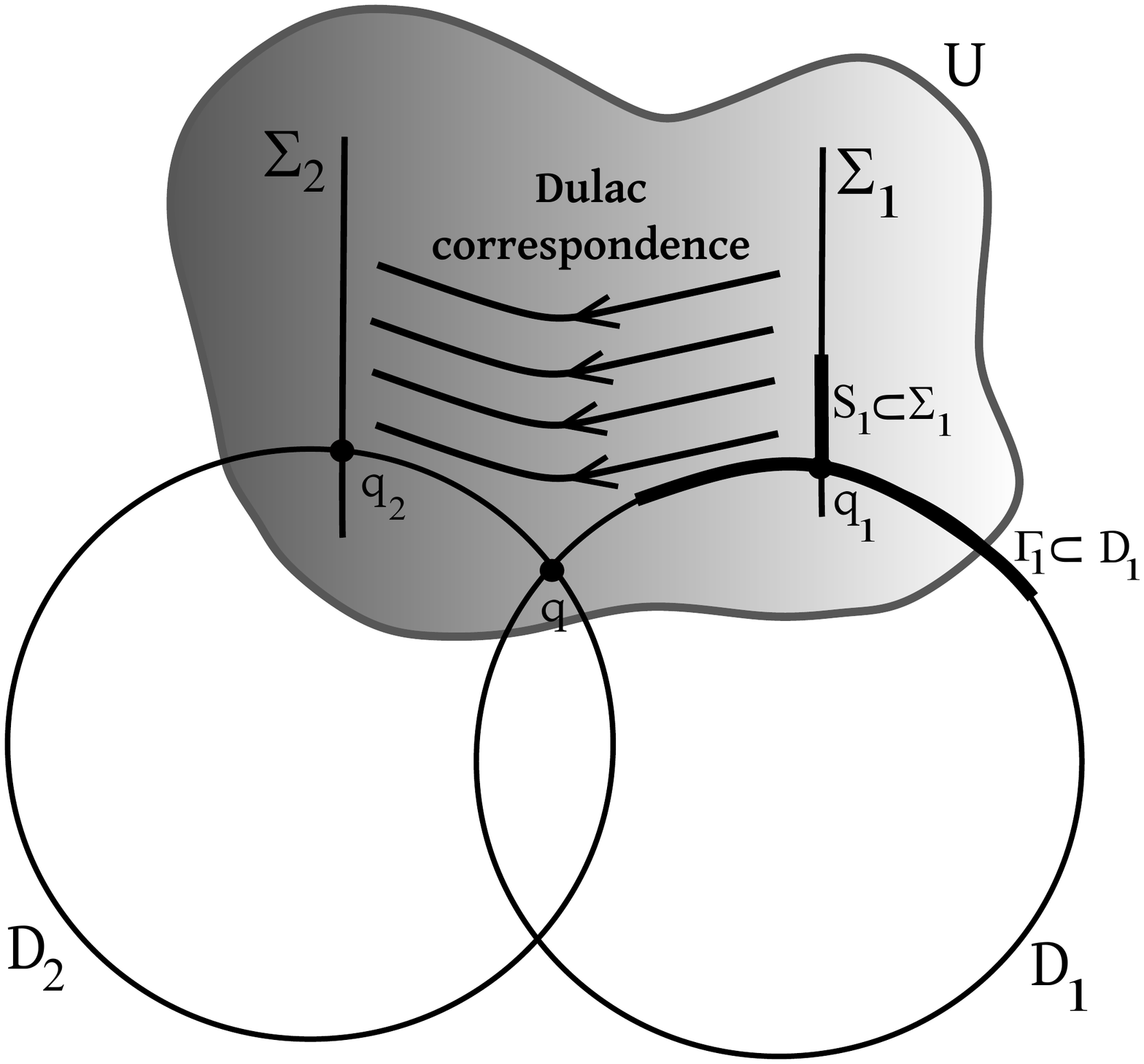} % leia abaixo
\caption{ }
\label{figura:PassingByProjective}
\end{figure}

\begin{proof}
Fix a point $q_{2} \in D_2  \cap U  \setminus \{q\}$. Let $\Sigma_{2} \subset U$ be a transverse section in $q_{2}$. Denote $\mathcal{F}(\Sigma_{1})$ the collection of subsets $E \subset \Sigma_{1}$ such that $E$ is contained in some leaf of $\widetilde{\mathcal{F}}|_{U}$. Define $\mathcal{F}(\Sigma_2)$ in a similar way. Denote by $\mathcal{D}_{q} :\mathcal{F}(\Sigma_{1}) \rightarrow \mathcal{F}(\Sigma_2)$ the Dulac correspondence given by $\mathcal{D}_q(x) = \{x^{m/n}\}$.  Let be the subcollection $\mathcal{A} = \{E \in \mathcal{F}(\Sigma_{1}) : E \cap S_{1} \neq \emptyset \}$. Take a sector $S_2 \subset \Sigma_{2}$ with vertex at $q_2$, such that $\mathcal{D}_q(E) \cap S_2 \neq \emptyset, \forall E \in \mathcal{A}$.

Define $\xi : S_2 \rightarrow S_{1}$ given by $\xi(y) = y^{n}$, where $n$ is given by the normal form. By the construction of $S_2$ and the Dulac correspondence, $\xi$ is a well define holomorphic function.  Define $\varphi_2:S_2 \rightarrow \mathbb{C}$ given by $\varphi_2 = \varphi_1 \circ \xi$. Since $\varphi_1$ is  a moderate sectorial first integral, we have that $\varphi_{1}(E) = cte$, for every $E_1 \in \mathcal{A}$. Then $\varphi_2 (\mathcal{D}_q(E)) = \varphi_1 \circ \xi (\mathcal{D}_q(E)) = \varphi_{1}(E) = cte$. Furthermore, for every element of holonomy $h \in \hat{Hol}(\widetilde{\mathcal{F}}, D_{1}, \Sigma_{1}) = \{f \in Diff(\Sigma_1,q_1)\ \tilde{L}_{z} = \tilde{L}_{f(z)} \textit{, for any } z \in (\Sigma_{1}, q_{1})\}$, we have a collection $\{h^{\mathcal{D}}\} \subset Diff(\Sigma_2,q_2)$ satisfying the adjunction equation $h^{\mathcal{D}} \circ \mathcal{D}_q = \mathcal{D}_q \circ h$. In particular, for every $y \in S_2$ the adjunction equation implies that $\xi \circ h^{\mathcal{D}}(y) = h \circ \xi (y)$. Therefore, for every $y \in S_2$, we have
\begin{eqnarray*}
\varphi_2 \circ h^{\mathcal{D}} (y) & = & \varphi_{1} \circ \xi \circ h^{\mathcal{D}}(y)  \\
                                  & = & \varphi_{1}  \circ h \circ \xi (y)               \\
                                  & = & (\varphi_{1}  \circ h) \circ \xi (y)             \\
                                  & = & \varphi_{1} \circ \xi (y)                        \\
                                  & = & \varphi_2(y).                              
\end{eqnarray*} 
It remains to prove the asymptotic property of $\varphi_2$. It is known that $\xi$ admits $\hat{\xi} = y^{n} = \xi$ as asymptotic expansion. Since $\varphi_{1}$ is a moderate sectorial first integral there is a nonzero formal power series $\hat{\varphi_{1}}$ such that $\varphi_{1}$ admits $\hat{\varphi_{1}}$ as asymptotic expansion in $S_{1}$. By Theorem \ref{BusAsymp}, $\varphi_2 = \varphi_{1} \circ \xi$ admits $\hat{\varphi_{1}} \circ \hat{\xi}$ as nonzero asymptotic expansion in $S_2$.
\end{proof}

Another important tool is the transport by holonomy. This tool will be required to carry the moderate sectoral first integral throughout (all over) the projective lines of the reduction of singularities.

\begin{Lemma} \label{LemmaTransHol}
Let $\widetilde{\mathcal{F}}$ be a foliation on a compact complex surface $\tilde{M}$, and  $D \subset \tilde{M}$ a compact (codimension one) invariant divisor with normal crossing and no triple points. We write $D = \cup_{j=1}^{m}D_{j}$, where each $D_{j}$ is an irreducible smooth component. Let be $q_1, q_2 \in D_j \setminus sing(\widetilde\fa)$, for some $j \in \{1, \dots, m\}$. Suppose that there are a transverse section $\Sigma_{1} \ni q_{1}$, a sector $S_1 \subset \Sigma_{1}$ with vertex at $q_{1}$, and a holomorphic function  $\varphi_1 : S_1 \rightarrow \mathbb{C}$, such that $\varphi_1$ is constant in the traces of each leaf of $\mathcal{F}$ in $S_1$ and which admits a nonzero asymptotic expansion in $S_1$. Then there are a  transverse section $\Sigma_{2} \ni q_2$, a sector $S_2 \subset \Sigma_{2}$ with vertex at $q_{2}$, and a holomorphic function  $\varphi_2 : S_2 \rightarrow \mathbb{C}$, such that $\varphi_2$ is constant in the traces of each leaf of $\mathcal{F}$ in ${S_2}$ and which admits a nonzero asymptotic expansion in $S_2$.
\end{Lemma}
\begin{figure}[h]
\centering % para centralizarmos a figura
\includegraphics[width=10cm]{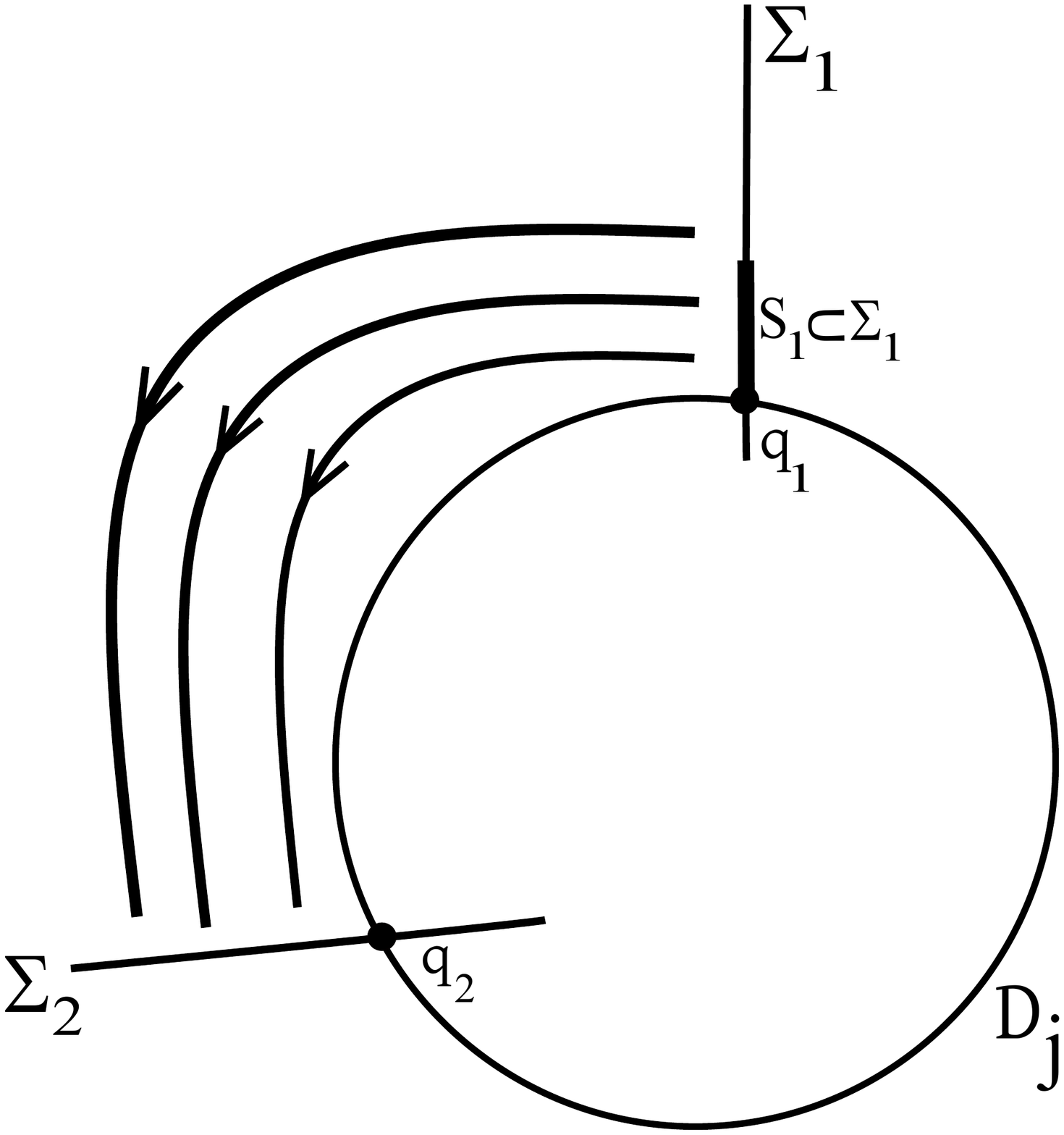} % leia abaixo
\caption{Transport by holonomy.}
\label{figura:PassingByProjective}
\end{figure}
\begin{proof}
Since $D_j$ is a 2-sphere we may choose a simple path $\gamma :[0,1] \rightarrow D_j \setminus sing(\tilde{\fa})$, such that $\gamma(0) = q_2$ and $\gamma(1) = q_1$. Take a transverse section $\Sigma_2 \ni q_2$. Then, we will consider the holonomy map of $\gamma$,  $f_{\gamma} : \Sigma_2 \rightarrow \Sigma_1$. Let $S_2$ be a sector in $\Sigma_2$ with vertex at $q_2$, such that $f_{\gamma}(S_2) \subseteq S_1$. Define $\varphi_2 : S_2 \rightarrow \mathbb{C}$ by $\varphi_2 = \varphi_1 \circ f_{\gamma}$. 
\begin{Claim} \label{claim1LemaHol}
$\varphi_2$ is invariant by $Hol(\mathcal{F},D_j,\Sigma_2, q_2)$
\end{Claim}
Indeed, let $g_{\delta} \in Hol(\mathcal{F},\Sigma_2,q_2)$ be the holonomy associated with a closed path $\delta : [0,1] \rightarrow D_j\setminus sing(\tilde{\fa})$, with $\delta(0) = \delta(1) = q_2$. Put $\beta = \gamma \circ \delta \circ \gamma^{-1}$. Then we can consider an element $g_{\beta} \in Hol(\mathcal{F},D_j,\Sigma_1,q_1)$ given by $g_{\beta} = f_{\gamma} \circ g_{\delta} \circ (f_{\gamma})^{-1}$. The fact that $\varphi_1$ is constant in the traces of each leaf implies that $\varphi_1 \circ g_{\beta} = \varphi_1$. Then
\begin{eqnarray*}
\varphi_1 &=& \varphi_1 \circ g_{\beta} \\
& = & \varphi_1 \circ (f_{\gamma} \circ g_{\delta} \circ (f_{\gamma})^{-1}) \\
& = &  (\varphi_1 \circ f_{\gamma}) \circ g_{\delta} \circ (f_{\gamma})^{-1} \\
& = &  \varphi_2 \circ g_{\delta} \circ (f_{\gamma})^{-1}. 
\end{eqnarray*}
From this we have that
\begin{eqnarray*}
\varphi_1 = \varphi_2 \circ g_{\delta} \circ (f_{\gamma})^{-1} &\Rightarrow & \varphi_1 \circ f_{\gamma} = \varphi_2 \circ g_{\delta} \\ &\Rightarrow & \varphi_2 = \varphi_2 \circ g_{\delta}.
\end{eqnarray*}
This proves Claim \ref{claim1LemaHol}.
\begin{Claim} \label{claim2LemaHol}
The function $\varphi_2$ admits a nonzero asymptotic expansion in $S_2$. 
\end{Claim}
In fact, by Theorem \ref{coefasymp2}, $f_{\gamma}$ admits $\hat{f_{\gamma}}(z) = \sum_{r=0}^{\infty}\frac{f_{r}}{r!}z^{r}$ as asymptotic expansion in the sector $S_2$, where $f_r = lim_{z \rightarrow 0}f^{(r)}(z)$ (recalling that $f^{(r)}$ is the r-th derivative of $f$). Since $lim_{z \rightarrow 0}f(z) \neq 0$ then $\hat{f_{\gamma}} \not\equiv 0$. Let $\hat{\varphi}$ be the nonzero asymptotic expansion of  ${\varphi}$. By Theorem \ref{BusAsymp}, $\varphi_2 = \varphi \circ f_{\gamma}$ admits $\hat{\varphi} \circ \hat{f_{\gamma}}$ as nonzero asymptotic expansion in $S_2$. This proves Claim \ref{claim2LemaHol}.

The separatrix $\Gamma_2$ divides $U$ in two parts $U = U_0 \cup U_1$ such that $q_0 \in U_0$ and $q_0 \not\in U_0$. Finally by Dulac correspondence (Lemma \ref{LemmaDulac}) there is a point $q_3 \in U_1$ such that there are a transverse section $\Sigma_{3} \ni q_{3}$, a sector $S_3 \subset \Sigma_{3}$ with vertex at $q_{3}$, and a holomorphic function  $\varphi_3 : S_3 \rightarrow \mathbb{C}$, such that $\varphi_3$ is constant in the traces of each leaf of $\mathcal{F}$ in $S_3$ and which admits a nonzero asymptotic expansion in $S_3$.
\end{proof}

\begin{Proposition} \label{NextEnough}
Let $\widetilde{\fa}$ be a foliation on a compact complex surface $\tilde{M}$, and  $D \subset \tilde{M}$ a compact (codimension one) invariant divisor with normal crossing, no cycles and no triple points. We write $D = \cup_{j=1}^{k}D_{j}$, where each $D_{j}$ is an irreducible smooth component. Let $p_k$ be a  point in $D_k \cap sing(\widetilde{\fa})$. Suppose that there is a separatrix such that $\Gamma_{k}\cap D_k = \{p_k\}$ and the pair $(\widetilde{\fa},\Gamma_k)$ admits a moderate sectorial first integral. Then for every non-singular point $p \in D_j \setminus sing(\widetilde{\fa})$ and any transverse section $\Sigma_p \ni p$ there is a sector $S_p \subset \Sigma_p$ with vertex at $p$ and a moderate sectorial first integral $\varphi_p :S_p \rightarrow \mathbb{C}$.
\end{Proposition}

\begin{figure}[h]
\centering % para centralizarmos a figura
\includegraphics[width=10cm]{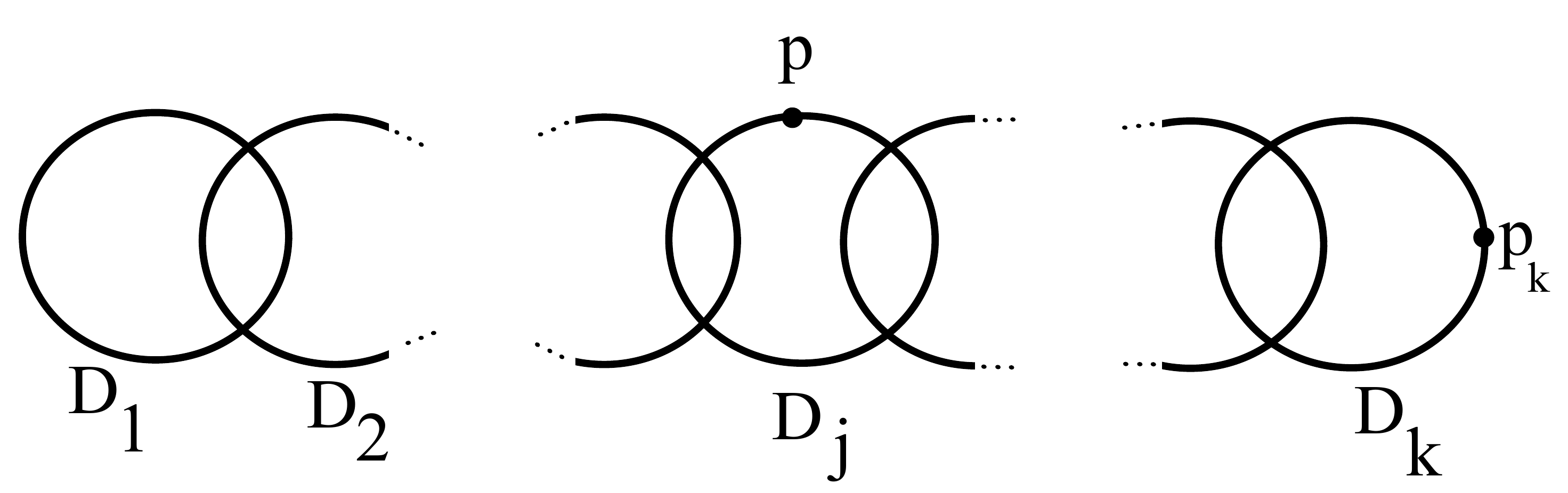} % leia abaixo
\caption{$p \in D_j \setminus sing(\widetilde{\fa})$ and $p_k \in D_k\cap sing(\widetilde{\fa})$.}
\label{figura:PropositionDulac}
\end{figure}

\begin{proof}
First we will suppose that each separatrix other than $\Gamma_k$ is contained in a projective line of $D$. Since $(\widetilde{\fa},\Gamma_k)$ admits a moderate sectorial first integral by Lemma \ref{LemmaDulac} there are a point $q_{k} \in D_k \setminus sing(\widetilde{\mathcal{F}})$, a  transverse section $\Sigma_{k} \ni q_{k}$, a sector $S_k \subset \Sigma_{k}$ with vertex at $q_{k}$, and a holomorphic function  $\varphi_k : S_k \rightarrow \mathbb{C}$, such that $\varphi_k$ is constant in the traces of each leaf of $\mathcal{F}$ in $S_k$ and which admits a nonzero asymptotic expansion in $S_k$. Let $t_k \in D_k \setminus sing(\widetilde{\fa})$ be a point close enough to  the corner $p_{k-1} \in D_{k-1} \cap D_{k}$. By Lemma \ref{LemmaTransHol} there are a  transverse section $\Sigma'_{k} \ni t_k$, a sector $S'_{k} \subset \Sigma'_{k}$ with vertex at $t_k$, and a holomorphic function  $\varphi'_k : S'_k \rightarrow \mathbb{C}$, such that $\varphi'_k$ is constant in the traces of each leaf of $\mathcal{F}$ in $S'_k$ and which admits a nonzero asymptotic expansion in $S'_k$. Consider a separatrix $\Gamma_{t_k} \subset D_k$ with $t_k \in \Gamma_{t_k}$. By Proposition \ref{INVfinite} the holonomy map by this separatrix is finite. Then by Theorem \ref{Theorem:analyticconjugation} there is a neighborhood $U_{k-1} \ni p_{k-1}$ such that $\widetilde{\fa}|_{U_{k-1}}$ assume the normal form. Then we apply again the Dulac correspondence, that is, by Lemma \ref{LemmaDulac} there are a point $q_{k-1} \in D_{k-1} \setminus sing(\widetilde{\mathcal{F}})$, a  transverse section $\Sigma_{k-1} \ni q_{k-1}$, a sector $S_{k-1} \subset \Sigma_{k-1}$ with vertex at $q_{k-1}$, and a holomorphic function  $\varphi_{k-1} : S_{k-1} \rightarrow \mathbb{C}$, such that $\varphi_{k-1}$is constant in the traces of each leaf of $\mathcal{F}$ in $S_{k-1}$ and which admits a asymptotic expansion in $S_{k-1}$. We can repeat this process until we have a transverse section $\Sigma_{p}$ with a sector $S_{p} \subset \Sigma_{p}$ and a moderate sectorial first integral $\varphi_p$ defined in $S_{p}$. 

Now we consider the case that there is a singularity $q \in D_i \cap sing(\widetilde{\fa})$ such that the separatrix $\Gamma_i$ passing through $q$ is not contained in $D_i$. From the previous case we obtain that there is neighborhood of $q$ that admits first integral. Then it is enough apply the Dulac correspondence to obtain that the pair $(\widetilde{\fa},\Gamma_i)$ admits a moderate sectorial first integral. Therefore, the proof following repeating the previous case until arrive to the point $p$.

\end{proof}

\section{Proof of Theorem \ref{Main}} \label{SectionMain}
We are now in conditions to prove our first main result.
\begin{proof}[\textbf{Proof of Theorem \ref{Main}}]
Let $\mathcal{F}$ be a non-dicritical germ of generalized curve at $0 \in \mathbb{C}^2$. Assume that some pair $(\mathcal{F}, \Gamma)$ admits a moderate sectorial first integral. We will use induction on the number $r$ of blow-ups of the resolution of $\fa$. 
\\
\\
\noindent \textbf{Case $r=0$}\\
We can divide this case in two:
\\
\item{\textbf{i)} \emph{Saddle-node case:}} 
By generalized curve hypothesis this case does not occur. 
\\
\item{\textbf{ii)} \emph{Non-degenerated case:}} 
We write $\mathcal{F}$ as $xdy - \lambda y dx + hot = 0$, with $\lambda \in \mathbb{C} \setminus \mathbb{Q}_{+}$ with invariant axes. It is known (\cite{Ilyashenko} pg. 70) that there are two transverse separatrices such that after a change of coordinates, we can assume that the separatrizes are the coordinate axes. We may assume $\Gamma : (y = 0)$ and a transverse section $\Sigma : (x = 1)$. The holonomy map of this separatrix is of the form
\[
h(y) = e^{2 \pi i \theta}y + ...
\] 
with $\theta \in \mathbb{R} \setminus \mathbb{Q}_{+}$. Let $\varphi$ be the moderate sectorial first integral admitting $\hat{\varphi}$ as non zero asymptotic expansion in a sector $S \subset \Sigma$ with vertex $p \in \Sigma \cup \Gamma$. Then $\varphi \circ h(y) = \varphi (y) $ in $S$. By Proposition \ref{INVfinite} $h$ is finite. By Theorem \ref{Theorem:analyticconjugation}, this implies that $\mathcal{F}$ is analytically linearizable as $m_{1}xdy + m_{2}ydx = 0$, with $m_{1}, m_{2} \in \mathbb{N}$. Therefore, there is a holomorphic first integral $F(x,y) = x^{m_2} + y^{m_1}$.
\\
\\
\noindent \textbf{Case $r \Rightarrow r + 1$}
\\
Suppose now that the theorem holds for foliations that can be reduced with $r > 0$ blow ups and let $\fa$ be reduced with $r+1$ blow-ups. We perform a first blow up $\pi_{1} : \widetilde{\mathbb{C}}_{0}^{2} \rightarrow \mathbb{C}^2$ obtaining a foliation $\fa(1) = \pi_1^{*}(\fa)$. Since $\fa$ is reduced with $r + 1$ blow-ups, each singularity $p_1 \in sing(\fa(1)) \subset \mathbb{P}_1$ can be reduced with at most $r$ blow ups. 
\begin{Claim} \label{ClaimA}
Given a singularity $p_1 \in sing(\fa(1)) \subset \mathbb{P}_1$ there is a separatrix $\Gamma_{p_1}$ of $\fa(1)$ through $p_1$ such that the pair $(\fa(1), \Gamma_{p_1})$ admits a moderate sectorial first integral.  
\end{Claim} 

\begin{figure}[h]
\centering % para centralizarmos a figura
\includegraphics[width=13cm]{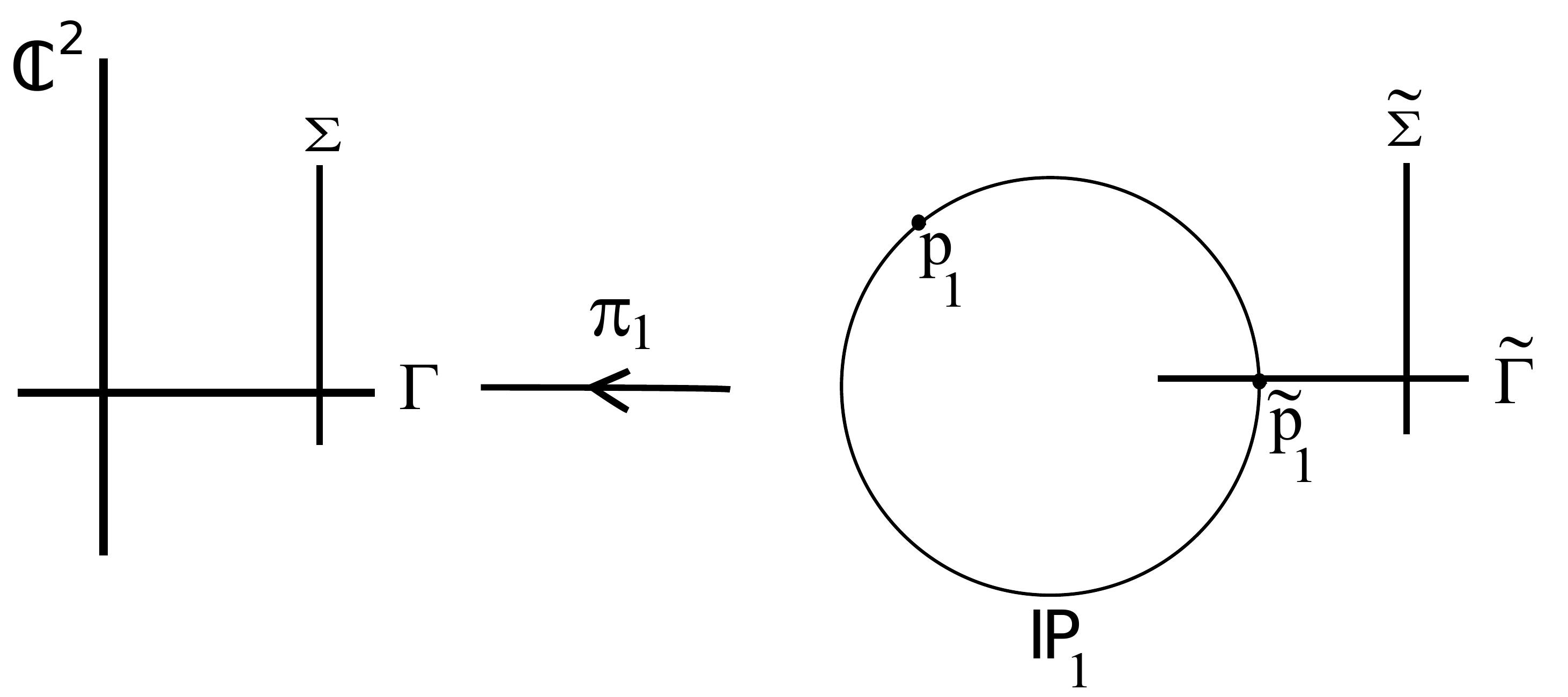} % leia abaixo
\caption{$\widetilde{\Gamma} = \pi_{1}^{*}(\Gamma)$ and $\widetilde{\Sigma} = \pi_{1}^{*}(\Sigma)$. }
\label{figura:FirstBlowUp}
\end{figure}
\begin{proof} [Proof of Claim \ref{ClaimA}]
The pull-back $\widetilde{\Gamma}$ of the separatrix $\Gamma$ is a separatrix of $\fa(1)$ for some singularity $\tilde{p_1}$ of $\fa(1)$ in $\mathbb{P}_1$. Furthermore, the pull-back of $\Sigma$ is a transverse section $\widetilde{\Sigma}$ to $\widetilde{\Gamma}$ where we can define a moderate sectorial first integral $\widetilde{\varphi} : \widetilde{S} \subset \widetilde{\Sigma} \rightarrow \mathbb{C}$ given by $\widetilde{\varphi} = \varphi \circ \pi_{1}$. This proves the claim for the singularity $\tilde{p_1}$.
\begin{figure}[h]
\centering % para centralizarmos a figura
\includegraphics[width=8cm]{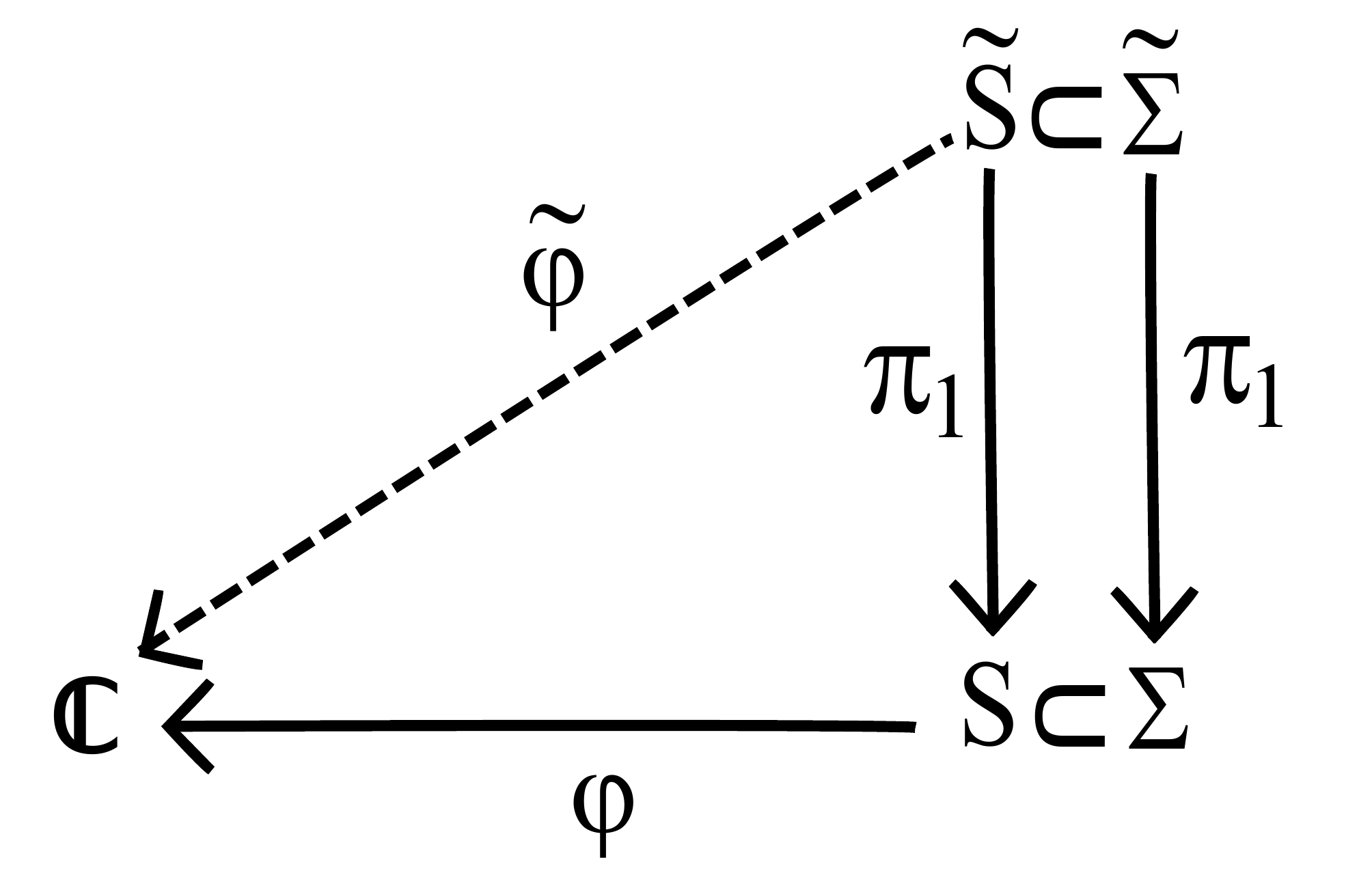} % leia abaixo
\caption{$\pi_1$ is a diffeomorphism outside of the origin}
\label{figura:diagram}
\end{figure}
Consider $p_1 \in sing(\fa(1)) \subset \mathbb{P}_1$. 

By Lemma \ref{LemmaDulac}, there are a point $q_{1} \in \mathbb{P}_1 \setminus sing(\fa(1))$, a  transverse section $\Sigma_{1} \ni q_{1}$, a sector $S_{1} \subset \Sigma_{1}$ with vertex at $q_{1}$, and a holomorphic function  $\varphi_{1} : S_{1} \rightarrow \mathbb{C}$, such that $\varphi_{1}$is constant in the traces of each leaf of $\mathcal{F}$ in $S_{1}$ and which admits a asymptotic expansion in $S_{1}$. Let us take a point $q_2$ taken sufficiently close to $p_1$.  By Lemma \ref{LemmaTransHol} there are a  transverse section $\Sigma_{2} \ni q_{2}$, a sector $S_{2} \subset \Sigma_{2}$ with vertex at $q_{2}$, and a holomorphic function  $\varphi_{2} : S_{2} \rightarrow \mathbb{C}$, such that $\varphi_{2}$ is constant in the traces of each leaf of $\mathcal{F}$ in $S_{2}$ and which admits a asymptotic expansion in $S_{2}$. Taking $\mathbb{P}_1$ itself as separatrix through $p_1$ the claim  is proved for any singularity in $\mathbb{P}_1$. This proves Claim \ref{ClaimA}.
\end{proof}

\begin{figure}[h]
\centering % para centralizarmos a figura
\includegraphics[width=13cm]{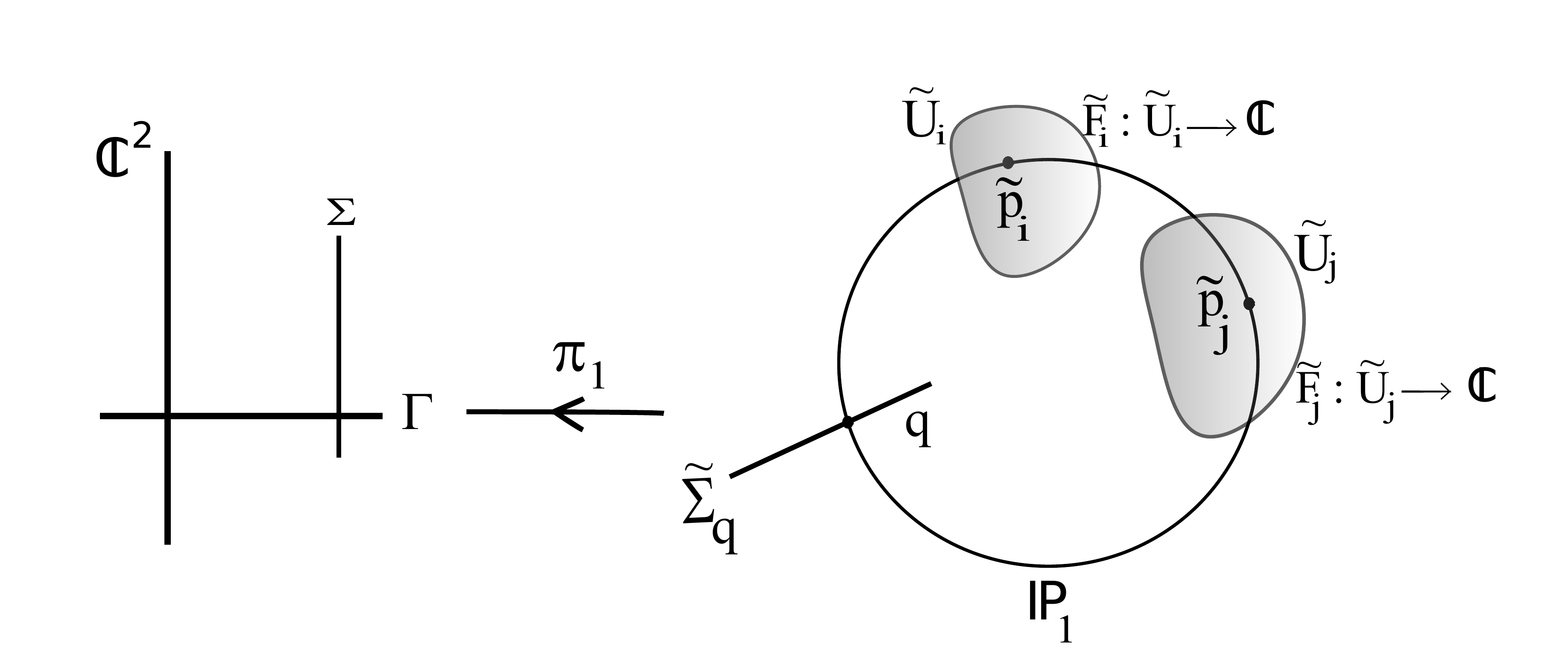} % leia abaixo
\caption{Local first integrals.}
\label{figura:FigBasic}
\end{figure}

By the induction hypothesis each singularity $\widetilde{p}_{j} \in sing(\widetilde{\fa}(1))$ admits a holomorphic first integral say $\widetilde{F}_j : \widetilde{U}_j \rightarrow \mathbb{C}$ defined in a neighborhood $\widetilde{p}_{j} \in \widetilde{U}_j \subset \widetilde{\mathbb{C}}^{2}_{0}$. Choose now a point $q \in \mathbb{P}(1) \setminus sing(\widetilde{\fa}(1))$ and a transverse disc $\widetilde{\Sigma}_q$ centered at $q$. Let $\widetilde{\varphi} : \widetilde{S} \subset \widetilde{\Sigma} \rightarrow \mathbb{C}$ be the inverse image of $\varphi: S \subset \Sigma \rightarrow \mathbb{C}$ by $\pi_1$. By holonomy map and Dulac correspondence (Lemmas \ref{LemmaDulac} and \ref{LemmaTransHol}) and recalling that $\widetilde{p}_1$ admits a holomorphic first integral, there exists a function $\widetilde{\varphi}_{q}: \widetilde{S}_q \rightarrow \mathbb{C}$ defined in a sector $\widetilde{S}_q  \subset \widetilde{\Sigma}_q $ which is a moderate sectorial first integral for $\widetilde{\fa}(1)$ at $q$ for the pair $(\widetilde{\fa}(1), \mathbb{P}_1)$. By Proposition \ref{INVfinite} the invariance group $Inv_{\widetilde{S}_q}(\widetilde{\varphi}_q)$ is finite. Denote by $\widetilde{f}_j$ the holonomy extension of $\widetilde{F}_j$ to the transverse section $\widetilde{\Sigma}_q$. Let $ Inv(\widetilde{f}_{1}, \dots, \widetilde{f}_{n}) < Diff(\widetilde{\Sigma}_q, q)$ be a subgroup generated by all the invariance groups of the $\widetilde{f}_i$.
\begin{Claim} \label{ClaimMainGroup}
$ Inv(\widetilde{f}_{1}, \dots, \widetilde{f}_{n}) $ is a finite group. 
\end{Claim}

Once we have proved Claim \ref{ClaimMainGroup} we can just follow the main steps in \cite{mattei-moussu}(pg. 504) in order to prove the existence of a holomorphic first integral for $\widetilde{\fa}(1)$ defined in a neighborhood of $\mathbb{P}(1)$  in $\widetilde{\mathbb{C}}_{0}^{2}$ and then we can finish the proof exactly as in \cite{mattei-moussu}. Therefore, it remains to prove Claim \ref{ClaimMainGroup}. 

\begin{proof}[Proof of Claim \ref{ClaimMainGroup}]
By construction each $\widetilde{f}_i$ is an extension of the local first integral $\widetilde{F}_j$ coming from the same function $\widetilde{\varphi}$. Then $\widetilde{\varphi}$ is constant in the levels of  $\widetilde{f}_i$ and $ Inv(\widetilde{f}_{1}, \dots, \widetilde{f}_{n}) \subset Inv_{\widetilde{S}}(\widetilde{\varphi})$. By Proposition \ref{INVfinite} $Inv_{\widetilde{S}}(\widetilde{\varphi})$ is finite. 
\end{proof}
\end{proof}

\section{Strong separatrix and sectorial first integral} \label{SaddlenodeSection}
In this section we will establish some results on saddle-node and give an important lemma for the proof of Theorem \ref{Main2}. We start by defining the notion of asymptotic expansion in dimension two. 
\begin{Definition}{\rm(\cite{stolovitch}, \cite{gerard-sibuya}, \cite{Ramis2}, \cite{Wasow})  \label{DefAsymp2var}
Let $\widehat{F}(x,y) = \sum_{i+j = 0}^{\infty}a_{ij}x^{i}y^{j} \in \mathbb{C}[[x,y]]$ be a formal power series, $S \subset \mathbb{C}$ a sector and $\Omega \subset \mathbb{C}$ an open set centered at $0 \in \mathbb{C}$. We shall say that a function $F(x,y)$ holomorphic in $\Omega \times S$, \textit{admits $\widehat{F}$ as asymptotic expansion} in $\Omega \times S$, if given any proper subsector $S'$ of $S$, any compact disc $\Omega'$ of $\Omega$ and for each $k\in \mathbb{N}$ there exist a constant $A_{S',\Omega',k} = A_k > 0$ such that
\begin{eqnarray*} \label{asymptoticExpress}
\mid F(x,y) - \sum_{i+j = 0}^{k - 1}a_{ij}x^{i}y^{j} \mid  <  A_{k}\mid\mid (x,y) \mid\mid ^{k}, \qquad \forall (x,y)\in \Omega' \times S'.
\end{eqnarray*}
}
\end{Definition}

Now we consider $\omega = 0$ a germ of saddle-node in $(\mathbb{C}^2,0)$. It is known \cite{Ramis2} that there exist numbers $p,\lambda$ and a formal coordinate change such that $\omega$ is formally equivalent to  
$$ \omega_{p,\lambda} (x,y)=
\begin{cases}
\dot{x} = x(1 + \lambda y^{p}) \\
\dot{y} = y^{p+1}. \\
\end{cases}
$$   
Moreover, there exist a formal power series $\hat{\varphi}(x,y) = x + \sum_{j=1}^{\infty}\varphi_{j}(x)y^j$ with $\varphi_{j}$ converging in some disc $\mid x\mid < R$, ($R>0$ do not dependent on $j$), such that the formal difeomorphism $\hat{\Phi}(x,y) = (\hat{\varphi}(x,y),y)$ satisfy $\hat{\Phi}^{*} \omega_{p,\lambda} \wedge \omega = 0$. We will denote $D_{p,\lambda}$ the subset of the saddle-node $\omega$ such that $w_{p,\lambda}$ is the normal final formal associated to $\omega$. From now on in this section we will use this notation. Under this considerations we have the following important Theorem.

\begin{Theorem} [Hukuara-Kimura-Matuda \cite{hukuara-kimura-matuda}] \label{Hukuara-Kimura-Matuda}
Let $\omega = 0$ be a saddle-node in $D_{p,\lambda}$ and $S$ be a sector in $\mathbb{C}$. If $S$ admits at most open $\frac{2\pi}{p}$, there is a limited holomorphic transformation
\[
\Phi :   \Omega \times S   \rightarrow \mathbb{C} \times  S
\]  
with $\Omega \subset \mathbb{C}$ is a neighborhood of $0 \in \mathbb{C}$, such that
 \item[(i)] $\Phi(x,y) = (\varphi(x,y),y)$
 \item[(ii)] $\Phi^{*} \omega_{p,\lambda} \wedge \omega = 0$
 \item[(iii)] $\varphi$ admits $\hat{\varphi}$ as asymptotic expansion in $\Omega \times S$ (in the sense of Definition \ref{DefAsymp2var}).
\end{Theorem}

\begin{Definition}{\rm
The holomorphic function $\Phi$ in Hukuara-Kimura-Matuda's theorem is called \textit{sectorial normalization}.
}
\end{Definition}
 
Let $\Gamma:\{y=0\}$ be the strong separatrix of $w_{p,\lambda}$ and $\Sigma :\{x=1\}$ be a transverse section to $\Gamma$. We will consider a sector 
$$S = \{y\in\Sigma\setminus \{(1,0)\}: Re(y)<0, -\frac{\pi}{p} < arg(y) < \frac{\pi}{p} \}.$$
Formally, we can do the following computations:
\begin{eqnarray*}
\frac{x(1 + \lambda y^{p})dy - y^{p+1}dx}{x^{p+1}x} = 0 \\
\Rightarrow \frac{dx}{x} - \left( \frac{1}{y^{p+1}} + \frac{\lambda}{y} \right) dy = 0 \\
\Rightarrow d \left( lnx - \lambda ln y + \frac{1}{p y^{p}} \right) = 0.
\end{eqnarray*}
Therefore, $F_{p,\lambda}(x,y) = x y^{-\lambda}e^{\frac{1}{p y^{p}}}$ is a first integral for $w_{p,\lambda}$ outside of $\Gamma$. Since the angle of $S$ is at most $\frac{2\pi}{p}$, by Hukuara-Kimura-Matuda theorem we have $\Phi$, $\varphi$ and $\Omega$ satisfying (i), (ii) and (iii) above. For simplicity of notation we may suppose $\Omega = \{x \in \mathbb{C} : |x| < 2\}$. Define $F: \Omega \times S\rightarrow \mathbb{C}$ given by $$F(x,y) = F_{p,\lambda} \circ \Phi (x,y).$$ 
By construction $F$ is constant in the traces of each leaf of $\omega=0$ in $S$.

\begin{Claim}
The asymptotic expansion of $F|_{\{x=1\} \times S}$ is zero in $S$.
\end{Claim}
\begin{proof}
By definition 
\begin{eqnarray*}
F|_{(\{x=1\} \times S)}(x,y) 
&=& F_{p,\lambda} \circ \Phi (1,y) \\
&=& F_{p,\lambda}(\varphi(1,y),y) \\
&=& \varphi(1,y)y^{-\lambda}e^{\frac{1}{p y^{p}}}.
\end{eqnarray*}
Then $F|_{(\{x=1\} \times S)}$ is holomorphic in $S$. Since $Re(y) < 0$ for all $y \in S$ then 
\begin{eqnarray*}
\mid y^{-\lambda}e^{\frac{1}{p y^{p}}}\mid &\leq & \mid y^{-\lambda}\mid \mid e^{\frac{1}{p y^{p}}}\mid \\ 
&=& \mid y^{-\lambda}\mid e^{Re\left(\frac{\bar{p y^{p}}}{\mid p y^{p}\mid^{2}}\right)}\\
&=& \mid y^{-\lambda}\mid e^{\frac{1}{p\mid y\mid^{2p}}Re\left({{y^{p}}}\right)}.
\end{eqnarray*}
Since $e^{\frac{1}{p\mid y\mid^{2p}}Re\left({{y^{p}}}\right)}\in \mathbb{R}$ and $Re(y)<0$, for any $k \in \mathbb{N}$ there is a constant $A_k >0$ such that  
\begin{eqnarray} \label{asympf1}
e^{\frac{1}{p\mid y\mid^{2p}}Re\left({{y^{p}}}\right)}&\leq & A_{k} \mid y \mid^k
\end{eqnarray}
with $y$ in a proper subsector $S'\subset S$. Then using item (iii) of Hukuara-Kimura-Matuda theorem and (\ref{asympf1}) we have
\begin{eqnarray*}
\mid F(1,y) \mid &\leq & \mid \varphi(1,y)y^{-\lambda}e^{\frac{1}{p y^{p}}} - 1 - \left(\sum_{j=1}^{k-1}\varphi_{j}(1)y^{j}\right)(y^{-\lambda}e^{\frac{1}{p y^{p}}})  +  1 + \left(\sum_{j=1}^{k-1}\varphi_{j}(1)y^{j}\right)(y^{-\lambda}e^{\frac{1}{p y^{p}}}) \mid  \\
&\leq & \left[  \mid \varphi(1,y) - 1 - \sum_{j=1}^{k-1}\varphi_{j}(1)y^{j}  \mid 
+ \mid  1 + \sum_{j=1}^{k-1}\varphi_{j}(1)y^{j} \mid \right] \mid  y^{-\lambda}e^{\frac{1}{p y^{p}}} \mid \\
&\leq & \left[   B_{k}\mid y \mid^{k} + C_{k}\mid y \mid^{k} \right] A_{k} \mid y \mid^k \\
&\leq & D_{k} \mid y \mid^{k}.
\end{eqnarray*}
\end{proof}
The above discussion allows us to enunciate the following lemma.
\begin{Lemma} \label{null asymptotic saddle-node}
Let $\fa$ be a germ of a saddle-node foliation at $0 \in \mathbb{C}^2$ and $\Gamma$ the strong separatrix of $\fa$. The pair $(\fa,\Gamma)$ admits sectorial first integral but not a moderate sectorial first integral.  
\end{Lemma}

The Lemma \ref{null asymptotic saddle-node} motivated the following definition. This notion is a central hypothesis in Theorem \ref{Main2}.
\begin{Definition} \rm{
A separatrix $\Gamma$ of $\mathcal{F}$ is called \textit{non-weak} if after the reduction of singularities of $\mathcal{F}$, every point of the exceptional divisor can be connected to $\Gamma$ by a sequence of projective lines starting at $\Gamma$ and such that every time we reach a corner, we arrive either on a separatrix of a non-degenerated singularity or on a strong (non-weak) manifold of a saddle-node. We refer to section §6 for further details on the resolution of singularities process.
}
\end{Definition}

\section{Proof of Theorem \ref{Main2}} \label{Proof of theorem2}
Following the above discussion we have.
\begin{proof}[\textbf{Proof of Theorem \ref{Main2}}]
Theorem \ref{Main2} is obtained similarly to Theorem \ref{Main} once we have proved the following proposition: 

\begin{Proposition} \label{PropTheoMain2}
Under the hypotheses of Theorem \ref{Main2}, there are no saddle-nodes in the reduction of singularities of $\fa$.
\end{Proposition}

\begin{proof}
Let $\widetilde{\fa} = \pi^{*}(\fa)$ be the reduction of singularities of $\fa$ and consider the exceptional divisor $D = \bigcup_{j=1}^{m} D_j$ where each $D_j$ is an irreducible component and it is diffeomorphic to an embedded projective line introduced as a divisor of the successive blowing up's. Assume that there is a saddle-node $q \in D_r$ for some $r \in\{1,\dots m\}$. Let $\widetilde{\Gamma} = \pi^{-1}(\Gamma)$ be a non-weak separatrix of $\widetilde{\fa}$. Since the pair $(\fa,\Gamma)$ admits a moderate sectorial first integral, the same holds to $(\widetilde{\fa},\widetilde{\Gamma})$. By hypothesis, every time we reach a corner $q \in \widetilde{\Gamma} \cap D_r$, we arrive through a strong in a strong manifold. Then, by Lemma \ref{null asymptotic saddle-node}, $\widetilde{\Gamma}$ can not be the strong separatrix of $q$. This implies that $q \not\in \widetilde{\Gamma}$. Consider the case where the strong separatrix of $q$ is contained in $D_r$. By hypothesis, every point of the exceptional divisor can be connected to $\widetilde{\Gamma}$ by a sequence of projective lines starting at $\widetilde{\Gamma}$. Thus, by Proposition \ref{NextEnough}, for every non-singular point $p \in D_r \setminus sing(\widetilde{\fa})$ and any transverse section $\Sigma_p \ni p$, there is a sector $S_p \subset \Sigma_p$ with vertex at $p$ and a moderate sectorial first integral $\varphi_p :S_p \rightarrow \mathbb{C}$. We can choose $p$ close enough to the point $q$ such that $p$ belongs to the strong separatrix of $q$, which is a contradiction by Lemma \ref{null asymptotic saddle-node}. 
\end{proof}
\end{proof}

\section{Proof of Theorem \ref{Main3}} \label{Section:GenCurve}
In this section we will show that if we consider all separatrices then the generalized curve hypothesis can be removed. For the convenience of the reader, we will define the Camacho-Sad Index \cite{Camacho-sad inv}.

Let $M$ be a complex surface. Consider an isolated singularity $q \in M$ of a foliation $\fa$ near $q$ induced by a holomorphic 1-form $\omega$. Suppose we have a separatrix $\Gamma$ of $\fa$ through $q$ and let $\phi: (\mathbb{C}^2,0) \rightarrow (\fa, q)$, be a local chart such that $\{\phi(x,0): x \in \mathbb{C}\} \subset \Gamma$. Then
$
(\phi^{*}\omega)(x,y) = A(x, y) dx + B(x, y) dy
$, where $A(0,0) = B(0,0) = 0$ and $A(x,0) = 0$.

\begin{Definition}{\rm (\cite{Camacho-sad inv} pg. 585)
We define the {\it index} of $\fa$ relative to $\Gamma$ at $q \in \Gamma$ by
\[
Ind_{q}(\fa,\Gamma)= - \Res_{x=0} \frac{\partial }{\partial y}\left(\frac{A}{B}\right)(x,0).
\]
}
\end{Definition}
\begin{Example}
{\rm
Let $q$ be an irreducible isolated singularity of a foliation $\fa$ induced by a holomorphic 1-form $\omega(x,y) = (\lambda_1 x + \dots)dy - (\lambda_2  y + \dots)dx$, with $\lambda_1, \lambda_2 \neq 0$. Fix $\Gamma_1,\Gamma_2$ the separatrices of $\fa$ in $q\in \Gamma_1,\Gamma_2$ tangent to the eigen-spaces of $\lambda_1$ and $\lambda_2$, respectively. Then  $Ind_{q}(\fa,\Gamma_1) = \frac{\lambda_2}{\lambda_1}$ and $Ind_{q}(\fa,\Gamma_2) = \frac{\lambda_1}{\lambda_2}$.
}
\end{Example}
\begin{Example}
{\rm
Let $p$ be a saddle-node isolated singularity of a foliation $\fa$ induced by a holomorphic 1-form $\eta(x,y) = \lambda_1 x + A(x,y)dy - y^{p+1}dx$, where $\lambda_1\neq 0$, and $A(x,y)$ with multiplicity $\geq 2$ in $(0,0)$. Fix $\Gamma:\{y=0\}$ the strong separatrix of $\fa$ in $p \in \Gamma$. Then  
\begin{eqnarray*}
Ind_{p}(\fa,\Gamma) = 0.
\end{eqnarray*}
}
\end{Example}

The \textit{Camacho-Sad Index Theorem} (Camacho-Sad \cite{Camacho-sad inv}, pg. 579) states that the sum of indexes of a foliation $\fa$  at all the
singularities in a compact analytic smooth invariant curve $\Ga$ on
a complex surface $M$ is equal to the self-intersection (first
Chern class) of $\Ga$ in $M$. In particular,
\[
\sum\limits_{p\in\sing(\fa)\cap\Ga}\,\,Ind_p(\fa,\Ga) = \Ga\cdot\Ga
\in \bz.
\]

Let $\pi: M_q \rightarrow M$ be the blowing up of a complex surface $M$ at a point $q \in M$ and let $\fa(1)$ be the foliation defined near $\mathbb{P}_q = \pi^{-1}(q)$ induced by $\pi^{*}\fa$. Denote $\widetilde{\Gamma} = \pi^{-1}(\Gamma - \{q\})$ and $\{m\} = \widetilde{\Gamma} \cap \mathbb{P}_q$. Then (\cite{Camacho-sad inv} pg. 585, Proposition 2.1)
\begin{eqnarray} \label{resiprop1}
Ind_{m}(\fa(1),\widetilde{\Gamma}) = Ind_{q}(\fa,\Gamma) -1,
\end{eqnarray}
and (\cite{Camacho-sad inv} pg. 586, Proposition 2.2)
\begin{eqnarray} \label{resiprop2}
\sum_{j=1}^{k}Ind_{q_{j}}(\fa(1),\mathbb{P}_q) = -1,
\end{eqnarray}
 where $q_{j} \in sing(\mathbb{P}_q)$ for all $j \in \{1,\dots k\}$.

Next, we present the proof of Theorem \ref{Main3}.  
\begin{proof}[\textbf{Proof of Teorem \ref{Main3}}]
We will construct our argumentation based on the number $r$ of blow-ups of the resolution of $\fa$.
\\
\\
\noindent \textbf{Case $r=0$:}\\
\emph{Saddle-node case:} 
By hypothesis, for every separatrix $\Gamma$ of $\fa$ the pair $(\fa,\Gamma)$ admits a moderate sectorial first integral. In particular, the same holds for the strong separatrix of the saddle-node, which is a contradiction by Lemma \ref{null asymptotic saddle-node}. Therefore, this case does not occur.
\\
\\
The \emph{non-degenerated case} follows as in Theorem \ref{Main}.
\\
\\
Assume that the result holds for foliations  that can be reduced with $r$ blow-ups. Let us give an insight of the general case by studying some additional cases. First we shall make an important remark. We perform the blow up $\pi_{r} : \widetilde{\mathbb{C}}_{0}^{2} \rightarrow \mathbb{C}^2$ obtaining a foliation $\fa(r) = \pi_1^{*}(\fa)$. Let $D = \bigcup_{j=1}^{r}\mathbb{P}_j$ be the  irreducible smooth components. 
\begin{Claim} \label{ClaimT31}
If there is a saddle-node $q \in \mathbb{P}_j$ in $\fa(r)$ then the strong separatrix $\widetilde\Gamma$ through $q$ is contained in $\mathbb{P}_j$.
\end{Claim}
\begin{proof} [Proof of Claim \ref{ClaimT31}]
Suppose by contradiction that the strong separatrix $\widetilde\Gamma$ through $q$ is transverse to $\mathbb{P}_1$. by Lemma \ref{null asymptotic saddle-node}, the strong separatrix of a saddle-node does not admit a moderate sectorial first integral. Then $\Gamma = \pi_{r}^{-1}(\widetilde\Gamma)$ is a separatrix to $\fa$, such that the pair $(\fa,\Gamma)$ does not admit a moderate sectorial first integral. However, by hypothesis, every par $(\fa,\cdot)$ admits a moderate sectorial first integral. This is a contradiction.  
\end{proof}

\noindent \textbf{Case $r = 1$:}\\
Suppose that $\fa$ is reduced with $r = 1$ blow-ups. We perform the blow up $\pi_{1} : \widetilde{\mathbb{C}}_{0}^{2} \rightarrow \mathbb{C}^2$ obtaining a foliation $\fa(1) = \pi_1^{*}(\fa)$. Let $\mathbb{P}_1$ be the  irreducible smooth component. 

Suppose by contradiction that there is a saddle-node $q$ in $\mathbb{P}_1$: 

\noindent Let $S$ be the strong separatrix trough $q$. By Claim \ref{ClaimT31} $S \subset \mathbb{P}_1$. By Camacho-Sad Index Theorem, $Ind_{q}(\fa(1),S) = 0$ and $\sum_{j=1}^{k}Ind_{q_j}(\fa(1),\mathbb{P}_1) = -1$. Hence, there is at least one singularity $q_{\alpha}\in \mathbb{P}_1 \cap sing(\fa(1))$ and a separatrix $\Gamma_{\alpha} \ni q_{\alpha}$ such that $\Gamma_{\alpha}$ is transverse to $\mathbb{P}_1$ and $q_\alpha$ is not a saddle-node. By hypothesis, the pair $(\fa(1),\Gamma_{\alpha})$ admits a moderate sectorial first integral. By Proposition \ref{NextEnough} for every non-singular point $p \in \mathbb{P}_1 \setminus sing(\fa(1))$ and any transverse section $\Sigma_p \ni p$, there is a sector $S_p \subset \Sigma_p$ with vertex at $p$ and a moderate sectorial first integral $\varphi_p :S_p \rightarrow \mathbb{C}$. We can choose $p$ close enough to saddle-node $q$ such that $p \in S$. By Lemma \ref{null asymptotic saddle-node} this is a contradiction. Therefore, there are no saddles in this case.
\\
\\
\noindent \textbf{Case $r = 2$:}\\
Suppose that $\fa$ is reduced with $r = 2$ blow-ups. We perform the blow up $\pi_2 : \widetilde{\mathbb{C}}_{0}^{2} \rightarrow \mathbb{C}^2$ obtaining a foliation $\fa(2) = \pi_2^{*}(\fa)$. Let $D = \bigcup_{j=1}^{2}\mathbb{P}_j$ be the  irreducible smooth components of $\fa(2)$. Suppose by contradiction that there is a saddle-node $q \in D$. Let $S$ be the strong separatrix trough $q$. We have two possibilities:
\\
\\
i) \textit{$q \notin \mathbb{P}_1 \cap \mathbb{P}_2$}: By Claim \ref{ClaimT31} $S \subset \mathbb{P}_1$ or $S \subset \mathbb{P}_2$. If $S \subset \mathbb{P}_1$ we proceed as in case $r = 1$. If $S \subset \mathbb{P}_2$ we observe that $Ind_{q}(\fa(2),S) = 0$ and $\sum_{j=1}^{k}Ind_{q_j}(\fa(2),\mathbb{P}_2) = -2$. Hence, there is at least one singularity $q_{\alpha}\in \mathbb{P}_2 \cap sing(\fa(2))$ and a separatrix $\Gamma_{\alpha} \ni q_{\alpha}$ such that $\Gamma_{\alpha}$ is transverse to $\mathbb{P}_2$ and $q_\alpha$ is not a saddle-node. By hypothesis, the pair $(\fa(1),\Gamma_{\alpha})$ admits a moderate sectorial first integral. By Proposition \ref{NextEnough} for every non-singular point $p \in \mathbb{P}_1 \setminus sing(\fa(1))$ and any transverse section $\Sigma_p \ni p$, there is a sector $S_p \subset \Sigma_p$ with vertex at $p$ and a moderate sectorial first integral $\varphi_p :S_p \rightarrow \mathbb{C}$. We can choose $p$ close enough to saddle-node $q$ such that $p \in S$. By Lemma \ref{null asymptotic saddle-node} this is a contradiction. Therefore, there are no saddles in this case.
\\
\\
ii) \textit{$q \in \mathbb{P}_1 \cap \mathbb{P}_2$}: By Claim \ref{ClaimT31} $S \setminus \{q\} \subset \mathbb{P}_1 - \mathbb{P}_2$ or $S \setminus  \{q\} \subset \mathbb{P}_2 - \mathbb{P}_1$. Following as previous case we prove that there is no saddle-node when $r=2$.
\\
\\
\noindent \textbf{Case $r = 3$:}\\
Suppose that $\fa$ is reduced with $r = 3$ blow-ups. We perform the blow up $\pi : \widetilde{\mathbb{C}}_{0}^{2} \rightarrow \mathbb{C}^2$ obtaining a foliation $\fa(3) = \pi^{*}(\fa)$. Let $D = \bigcup_{j=1}^{3}\mathbb{P}_j$ be the  irreducible smooth components such that $\mathbb{P}_1 \cap \mathbb{P}_3 \neq \emptyset$ and $\mathbb{P}_2 \cap \mathbb{P}_3 \neq \emptyset$. Then we have three possibilities:\\
\\
i) \textit{There are no saddle-nodes on the corners}: Let $q_k \in \mathbb{P}_k$ be a saddle-node in some irreducible component $\mathbb{P}_k$ and $S_k$ the strong separatrix through $q_k$. It follows from Claim \ref{ClaimT31} that $S_k \subset \mathbb{P}_k$. Let $\Gamma$ be the separatrix of Camacho-Sad \cite{Ca-Sad-LN generalizedCurve}. By hypothesis, the pair $(\fa(3), \Gamma)$ admits a moderate sectorial first integral. Since $q_k$ is not a corner then $q_k \not\in \Gamma \cap \mathbb{P}_k$. By Proposition \ref{NextEnough} for every non-singular point $p \in \mathbb{P}_k \setminus sing(\fa(3))$ and any transverse section $\Sigma_p \ni p$, there is a sector $S_p \subset \Sigma_p$ with vertex at $p$ and a moderate sectorial first integral $\varphi_p :S_p \rightarrow \mathbb{C}$. We can choose $p$ close enough to saddle-node $q_k$ such that $p \in S_k$. By Lemma \ref{null asymptotic saddle-node} this is a contradiction. Therefore, there are no saddle-nodes in this case.
\\
\\
ii) \textit{Only one corner has a saddle-node}: Without loss of generality, we will suppose that the saddle-node $q \in \mathbb{P}_1 \cap \mathbb{P}_3$. By Claim \ref{ClaimT31} the strong separatrix $S \ni q$ is contained in $\mathbb{P}_1$ or $\mathbb{P}_3$. Let $\Gamma$ be the separatrix of Camacho-Sad \cite{Ca-Sad-LN generalizedCurve}. If $S \subset \mathbb{P}_1$ we use the local coordinates of $\mathbb{P}_1$ for $\Gamma$. By hypothesis, the pair $(\fa(3), \Gamma)$ admits a moderate sectorial first integral. Since the corner $\mathbb{P}_2 \cap \mathbb{P}_3$ is not a saddle-node it follows from Proposition \ref{NextEnough} that for every non-singular point $p \in \mathbb{P}_1  \setminus sing(\fa(3))$ and any transverse section $\Sigma_p \ni p$, there is a sector $S_p \subset \Sigma_p$ with vertex at $p$ and a moderate sectorial first integral $\varphi_p :S_p \rightarrow \mathbb{C}$. We can choose $p$ close enough to saddle-node $q$ such that $p \in S$. By Lemma \ref{null asymptotic saddle-node} this is a contradiction. If $S \subset \mathbb{P}_3$ we use the local coordinates of $\mathbb{P}_2$ for $\Gamma$. Since the corner $\mathbb{P}_2 \cap \mathbb{P}_3$ is not a saddle-node we can use the Proposition \ref{NextEnough} and Lemma \ref{null asymptotic saddle-node} we have a contradiction. Therefore, there are no saddle-nodes in this case.
\\
\\
iii) \textit{The two corners have saddle-nodes}: Let $q_1 \in \mathbb{P}_1 \cap \mathbb{P}_3$ and $q_2 \in \mathbb{P}_2 \cap \mathbb{P}_3$ be the saddle-nodes and $S_j$ be the strong separatrix trough $q_j$ with $j \in \{1,2\}$. By Claim \ref{ClaimT31} $S_1 \subset \mathbb{P}_3$ or $S_1 \subset \mathbb{P}_1$ and $S_2 \subset \mathbb{P}_2$ or $S_2 \subset \mathbb{P}_3$. Suppose that  $S_1,S_2 \subset \mathbb{P}_3$. Then  $Ind_{q_j}(\fa(3),S_j) = 0$ for $j =1,2$. Furthermore, $\sum_{j=1}^{k}Ind_{q_j}(\fa(3),\mathbb{P}_3) = -3$, where $q_j \in \mathbb{P}_3 \cap sing(\fa(3))$. Hence, there is at least one singularity $q_{\alpha}\in \mathbb{P}_3 \cap sing(\fa(3))$ and a separatrix $\Gamma_{\alpha} \ni q_{\alpha}$ such that $\Gamma_{\alpha}$ is transverse to $\mathbb{P}_3$ and $q_\alpha$ is not a saddle-node. By hypothesis, the pair $(\fa(3),\Gamma_{\alpha})$ admits a moderate sectorial first integral. By Proposition \ref{NextEnough}, for every non-singular point $p \in \mathbb{P}_1 \setminus sing(\fa(1))$ and any transverse section $\Sigma_p \ni p$, there is a sector $S_p \subset \Sigma_p$ with vertex at $p$ and a moderate sectorial first integral $\varphi_p :S_p \rightarrow \mathbb{C}$. For $j =1,2 $ we can choose $p$ close enough to saddle-node $q_j$ such that $p \in S_j$. By Lemma \ref{null asymptotic saddle-node} this is a contradiction. If $S_1 \subset \mathbb{P}_1$ or $S_2 \subset \mathbb{P}_2$ we can apply Proposition \ref{NextEnough} and Lemma \ref{null asymptotic saddle-node} we have a contradiction (as in the first arguments of the case (ii) above). Therefore, there are no saddle-nodes in this case.
\\
\\
\noindent \textbf{Case $r =4$:}\\
Suppose that $\fa$ is reduced with $r = 4$ blow-ups. We perform the blow up $\pi : \widetilde{\mathbb{C}}_{0}^{2} \rightarrow \mathbb{C}^2$ obtaining a foliation $\fa(4) = \pi^{*}(\fa)$. Let $D = \bigcup_{j=1}^{4}\mathbb{P}_j$ be the irreducible smooth components. We have the following possibilities:
\\
\\
i) \textit{There are no saddle-nodes on the corners}: Repeat the arguments of (i) case $r=3$.
\\
\\
ii) \textit{Only one corner has a saddle-node}: Repeat the arguments of (ii) of case $r=3$.
\\
\\
iii) \textit{The two corners of the same projective line are saddle-nodes}: Repeat the arguments of (iii) case $r=3$.
\\
\\
iv) \textit{Only one corner of each projetctive line is a saddle-node}:
\begin{figure}[h]
\centering % para centralizarmos a figura
\includegraphics[width=7cm]{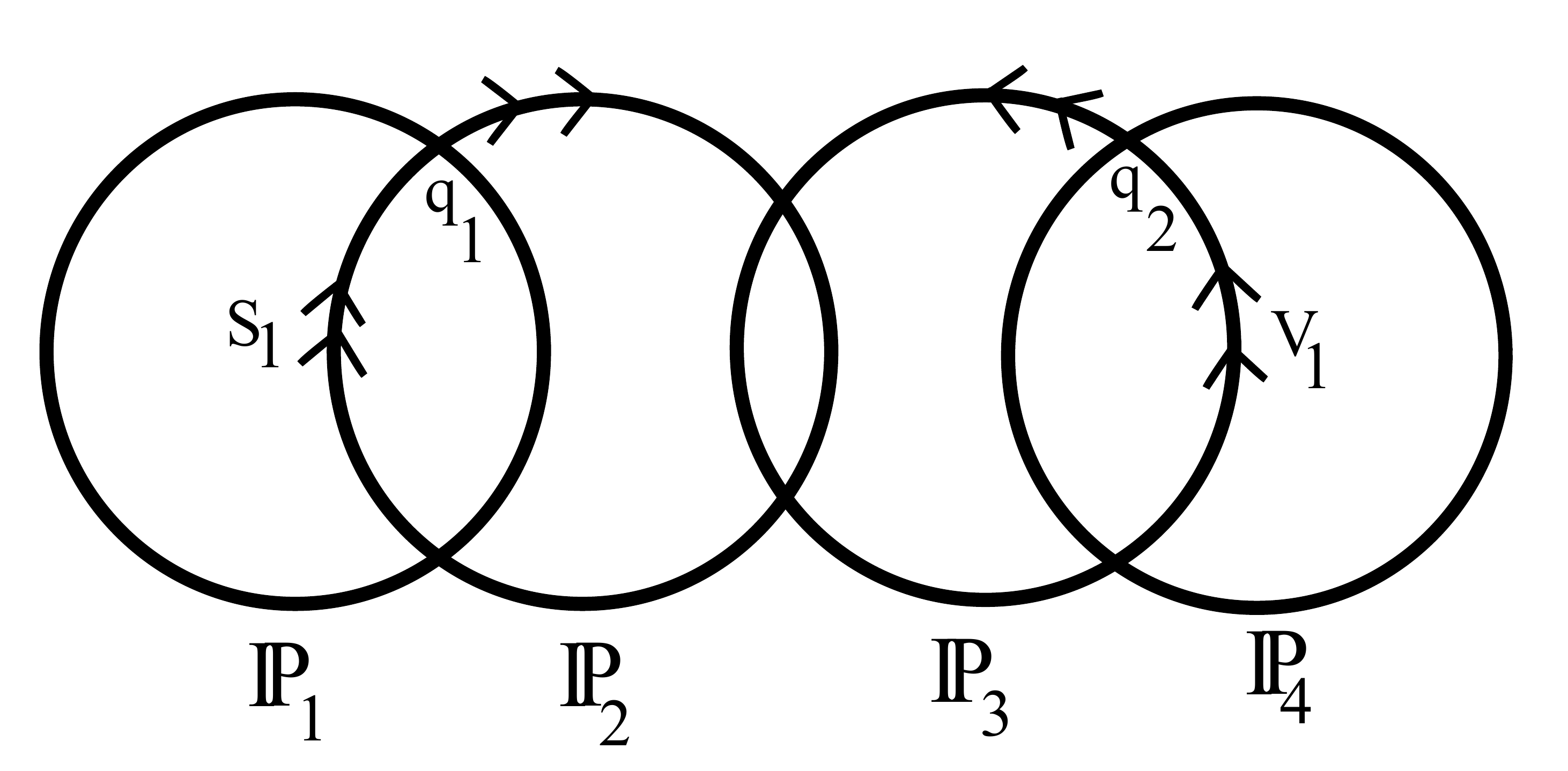} % leia abaixo
\caption{Strong separatrices $S_1 \subset \mathbb{P}_2$ and $V_1\subset \mathbb{P}_3$.}
\label{figura:case41}
\end{figure}
Without loss of generality, we consider $\mathbb{P}_1$ the projective line containing local coordinate chards of the Y-axis and $\mathbb{P}_4$ the projective line containing local coordinate chards of the X-axis.
Also, we consider $\mathbb{P}_1 \cap \mathbb{P}_2 \neq \emptyset$,  $\mathbb{P}_2 \cap \mathbb{P}_3 \neq \emptyset$ and  $\mathbb{P}_3 \cap \mathbb{P}_4 \neq \emptyset$. Suppose that there are two saddles-nodes  $q_1 \in \mathbb{P}_1 \cap \mathbb{P}_2$, $t_1 \in \mathbb{P}_3 \cap \mathbb{P}_4$ and a resonant saddle $p \in \mathbb{P}_2 \cap \mathbb{P}_3$. Let $S_1 \ni q_1$ and $V_1 \ni t_1$ be the stronger separatrices through $q_1,t_1$ respectivelly. It follows from Claim \ref{ClaimT31} that $S_1,V_1 \subset D$. If $S_1 \subset \mathbb{P}_1$ and $V_1 \subset \mathbb{P}_4$ we repeat the same arguments of cases $r=1,2$ and we came to a contradiction. Suppose now that $S_1 \subset \mathbb{P}_2$ and $V_1 \subset \mathbb{P}_3$ (see figure \ref{figura:case41}). Let $\widetilde{\omega}_2$ be a 1-form associated to $\fa(3)$ in $\mathbb{P}_2$, given by $\widetilde{\omega}_2 = mxdy + nydx + \dots = 0$, where $m,n \in \mathbb{N}$ and $<m,n> = 1.$ Let $\Gamma_2 \subset \mathbb{P}_2$ be the separatrix through the resonant saddle $p \in \Gamma_2$. The holonomy map by $\Gamma_2$ is of the form
\[
h_{p}(y) = e^{2 \pi i \frac{m}{n}} y + ...
\] 
Let $q_1,\dots, q_k$ be all saddle-nodes in $\mathbb{P}_2$. We consider the holonomy maps $h_{q_1}, \dots ,h_{q_k}$ by each strong separatrix $S_1, \dots S_k$ through $q_1,\dots, q_k$, respectively. Since $\mathbb{P}_2$ is simply connected, we have 
$$h_{p} \circ h_{q_1} \circ \dots \circ h_{q_k} = Id.$$
Then, if each $h_{q_j}$ is tangent to identity then $h_p$ is also tangent to identity. However, $h'_{p}(0) = \frac{m}{n}$ with $m,n \in \mathbb{N}$ and $<m,n> = 1$. Hence, $m = n$. By Camacho-Sad Index Theorem 
\[
Ind_p(\fa(4),\Gamma_2) + \sum_{j=1}^{k} Ind_{q_j}(\fa(4),S_j) = - \alpha
\]
where $\alpha \in \mathbb{N}$. Since $\sum_{j=1}^{k} Ind_{q_j}(\fa(4),S_j) = 0$ then 
$$Ind_p(\fa(4), \Gamma_2) = -\alpha.$$
On the other hand, $$Ind_p(\fa(4), \Gamma_2) = - \frac{m}{n} = - 1 \Rightarrow -\alpha = -1.$$ This implies that, $\mathbb{P}_2\cdot \mathbb{P}_2 = -1$. Analogously, $\mathbb{P}_3 \cdot \mathbb{P}_3 = -1.$ By equation (\ref{resiprop1}), this is a contradiction. \\
\\
\noindent \textbf{Case $r > 4$:} All other cases follows from the above.

Therefore, under this hypothesis there are no saddle-nodes in the resolution. The proof follows from Theorem \ref{Main}.
\end{proof}

\section{Proof of Theorem \ref{Cor3}}
As an application of our results we will give a proof of the following theorem of Mattei-Moussu. 
\\
\noindent{\textbf{Theorem 1.4.}}  (Mattei-Moussu \cite{mattei-moussu}, pg. 472, Theorem A)
Let $\mathcal{F}$ be a germ of holomophic foliation at $0 \in \mathbb{C}^2$. Suppose that $\fa$ admits a nonzero formal first integral. Then $\fa$ admits a holomorphic first integral.
\begin{proof}
We shall denote by $\fa^*$ the  foliation
with isolated  singularities  $\fa^*=\pi^*(\om)$. Thus $\fa^*$ is
the pull-back of $\fa$ via the blow-up map $\pi\colon \tilde
{\mathbb C^2_0} \to \mathbb C^2$. We can write the Taylor expansion
of $\om$ at $0$ as:
$$\om=\sum_{j=k}^{\infty}(P_jdy-Q_jdx),$$
where $P_j$ and $Q_j$ are homogeneous polynomials of degree $j$,
with $P_k \not\equiv 0$ or $Q_k \not\equiv 0$. The $1$-form
$\pi^*(\om)$ writes in the chart  $((t,x),U)$ as:
$$\pi^*(\om)= \sum_{j=k}^{\infty}(P_j(x,tx)d(tx)-Q_j(x,tx)dx)=$$
$$=x^k.\sum_{j=k}^{\infty} x^{j-k}.[(tP_j(1,t)-Q_j(1,t))dx - xP_j(1,t)dt].$$
Dividing the above $1$-form by
$x^k$ we obtain:
$$(*)\ \ x^{-k}.\pi^*(\om)= (tP_k(1,t)-Q_k(1,t))dx + xP_k(1,t)dt + x.\alpha$$
where
$\alpha=\sum_{j=k+1}^{\infty}x^{j-k-1}.[(tP_j(1,t)-Q_j(1,t))dx+xP_j(1,t)dt]$.
Set  $R(x,y)=yP_k(x,y)-xQ_k(x,y)$, in such a way that
$x^{-k}.\pi^*(\om)=R(1,t)dx +xP_k(1,t)dt+x.\alpha$.
\begin{Claim} \label{Claim1}
$\fa$ is non-dicritical.
\end{Claim}
Let $\widehat{F}(x,y) = \sum_{i+j = 0}^{\infty}a_{ij}x^{i}y^{j}$ be a formal first integral to $\fa$. There is a formal power series $\widehat{G} \in \mathbb{C}[[x,y]]$ such that $\om = \widehat{G}d\widehat{F}$. We can consider $\om = d\widehat{F}$, since $d\widehat{F}$ defines the same foliation as $\widehat{G}d\widehat{F}$. Write $$\widehat{F} = \widehat{F}_{k+1} + \widehat{F}_{k+2} + \widehat{F}_{k+3} + \dots $$  where $\widehat{F}_{\lambda}(x,y) = \sum_{i+j =\lambda}a_{ij}x^{i}y^{j}$ and $\widehat{F}_{k+1} \not\equiv 0$. Then
$$
d\widehat{F} = d\widehat{F}_{k+1} + d\widehat{F}_{k+2} +  d\widehat{F}_{k+3} + ...
$$
We have
\begin{eqnarray*}
\om_{k+1} &=& d\widehat{F}_{k+1} \\
          &=& \sum_{\mid I\mid = k+1} a_I q_1 x^{q_1 - 1} y^{q_2} dx + \sum_{\mid I \mid = k+1}a_Iq_2x^{q_1}y^{q_2}    
\end{eqnarray*}
Let $\overrightarrow{V}(x,y) = x \frac{\partial}{\partial x} + y \frac{\partial}{\partial y}$ be the radial vector field in $\mathbb{C}^2$. Then 
\begin{eqnarray*}
R(x,y) &=& \om \cdot \overrightarrow{V}  \\
       &=& d\widehat{F}_{k+1}  \cdot \overrightarrow{V} \\
       &=& \sum_{\mid I\mid = k+1} a_I q_1 x^{q_1} y^{q_2} + \sum_{\mid I\mid = k+1} a_I q_2 x^{q_1 - 1} y^{q_2}\\
       &=& (k+1)\sum_{\mid I\mid = k+1} a_I q_1 x^{q_1} y^{q_2}\\
       &=& (k+1) \widehat{F}_{k+1}.
\end{eqnarray*}
Since $\widehat{F}_{k+1} \not\equiv 0$ therefore $R \not\equiv 0$. This proves Claim \ref{Claim1}.

\begin{Claim} \label{Claim2}
$\fa$ is a generalized curve.
\end{Claim}
Indeed, by blowing-up each singularity in the reduction process also admits a formal first integral. Therefore, no saddle-nodes appear in the reduction of singularities.
\\
\\ 
From now on we consider the exceptional divisor $D = \bigcup_{j=1}^{m} D_j$ where each $D_j$ is irreducible component is diffeomorphic to an embedded projective line introduced as a divisor of the successive blowing up's. Fix $\Gamma : \{y=0\}$ a separatrix of $\fa$. Without loss of generality, we will suppose that $\widetilde{\Gamma} = \pi^{-1}(\Gamma)$ is transverse to the projective line $D_1$ and $\widetilde{\Gamma} : \{t = 0\}$ on the local coordinate chart $E(x,t)= (x,xt)$ of $D_1$. 
\begin{figure}[h]
\centering % para centralizarmos a figura
\includegraphics[width=8cm]{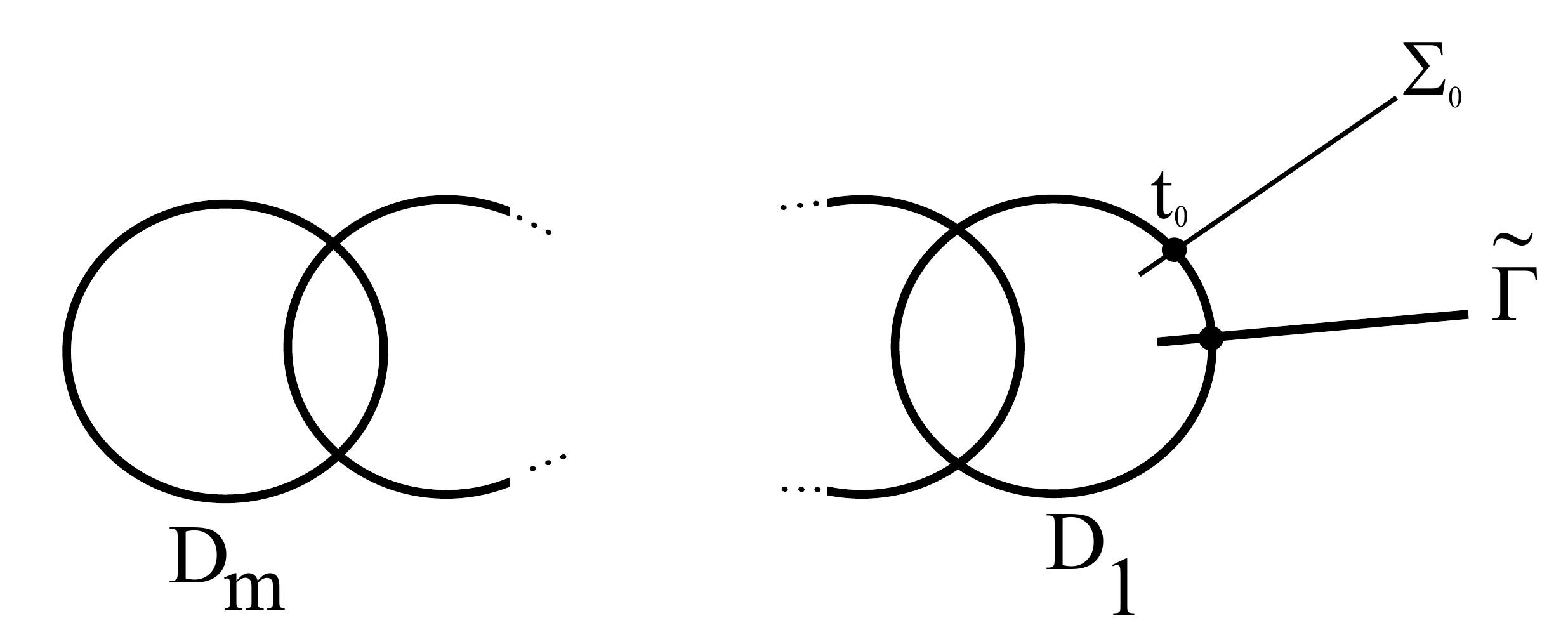} % leia abaixo
\caption{ }
\label{figura:applic1}
\end{figure}
Consider a local coordinate chart $E(x,t)= (x,xt)$. Then
\begin{eqnarray*}
\widehat{F} \circ E(x,t) & = & \sum_{i+j = 0}^{\infty}a_{ij}x^{i}(tx)^{j}\\
                         & = & \sum_{i+j = 0}^{\infty}a_{ij}t^{j}x^{i+j}\\
                         & = & \sum_{k = 0}^{\infty}\left(\sum_{i=0}^{k} a_{i(k-i)}t^{k-i}\right) x^{k}\\
                         & = & \sum_{k = 0}^{\infty}A_k(t)x^{k},
\end{eqnarray*}
where the coefficients $A_k$ are polynomials at $t$. Fix a non singular point $t_0 \in D_1 \setminus sing(\widetilde{\fa})$ next enough to the corner $\widetilde{\Gamma} \cap D_1$. Then $\widehat{F} \circ E(x,t_0) =\sum_{k = 0}^{\infty}A_k(t_0)x^{k} \in \mathbb{C}[[x]]$, that is, a formal power series in one variable. Let $\Sigma_0$ be a transverse section at $t_0$. By Borel-Ritt Theorem, given a sector $S_0 \subset \Sigma_0$ there is a holomorphic function $\varphi_0:S_0 \rightarrow \mathbb{C}$ that admits $\widehat{F} \circ E(x,t_0)$ as asymptotic expansion in $S_0$. By construction, $\varphi_0 :S_0 \rightarrow \mathbb{C}$ is constant in the trace of each leaf of $\widetilde{\fa}$ in $S_0$. Furthermore, $\varphi_0$ admits a nonzero asymptotic expansion. Since $\fa$ is a generalized curve, the corner $D_1 \cap \widetilde{\Gamma}$ can not be a saddle-node. By Proposition \ref{NextEnough} the pair $(\widetilde{\fa},\widetilde{\Gamma})$ admits a moderate sectorial first integral. By isomorphism  $\pi:\widetilde{\fa} \setminus D \rightarrow \fa \setminus \{0\}$ we conclude that the pair $(\fa,\Gamma)$ admits a moderate sectorial first integral. Recalling that $\fa$ is non-dicritical and a generalized curve we can apply Theorem \ref{Main} and we conclude that $\fa$ admits holomorphic first integral. 
\end{proof}

\bibliographystyle{amsalpha}

\end{document}